\documentclass[11pt]{amsart}
\usepackage[margin=2.5cm]{geometry}
\usepackage{amscd}      
\usepackage{amssymb}
\usepackage{amsmath, amsthm, graphics}
\usepackage[all]{xy}
\xyoption{2cell} \UseAllTwocells \xyoption{frame} \CompileMatrices
\usepackage{cite}
\usepackage{pgf}
\usepackage{tikz}
\usetikzlibrary{arrows,automata}

\allowdisplaybreaks[3]
\usepackage{latexsym}
\usepackage{epsfig}
\usepackage{amsfonts}
\usepackage{enumerate}
\usepackage{times}
\usepackage{hyperref}
\usepackage[utf8]{inputenc}
\usepackage[russian,english]{babel}

\newtheorem{theorem}{Theorem}[section]
\newtheorem{corollary}[theorem]{Corollary}
\newtheorem{lemma}[theorem]{Lemma}

\newtheorem{proposition}[theorem]{Proposition}
\newtheorem{conjecture}[theorem]{Conjecture}

\newtheorem*{claim*}{Claim}

\theoremstyle{remark}
\newtheorem{remark}[theorem]{Remark}

\newcommand{\bpsi}{\bar{\psi}}
\newcommand{\bbpsi}{\bar{\bar{\psi}}}
\newcommand{\Mbar}{\overline{\M}}
\newcommand{\proj}{\mathbb{P}}

\renewcommand\H{{\mathcal H}}  

\newcommand{\M}{\mathcal{M}}

\newcommand{\A}{\mathbb{A}}

\newcommand{\cE}{\mathcal{E}}
\newcommand{\cD}{\mathcal{D}}
\newcommand{\cF}{\mathcal{F}}

\newcommand{\bt}{\mathbf{t}}

\newcommand{\cH}{\mathcal{H}}

\newcommand{\eq}{\text{\textit{eq}}}

\newcommand{\tw}{\text{\textit{tw}}}
\newcommand{\block}{\text{\textit{block}}}

\def\<{\left\langle}
\def\>{\right\rangle}

\newcommand{\NN}{\mathbb{N}}
\newcommand{\QQ}{\mathbb{Q}}

\newcommand{\VV}{\mathbb{V}}
\newcommand{\RR}{\mathbb{R}}
\newcommand{\CC}{\mathbb{C}}
\newcommand{\Cstar}{\CC^\times}
\newcommand{\ZZ}{{\mathbb{Z}}} 

\newcommand{\bx}{\mathbf{x}}
\newcommand{\bepsilon}{\boldsymbol{\epsilon}}
\newcommand{\by}{\mathbf{y}}
\newcommand{\cA}{\mathcal{A}}
\newcommand{\cI}{\mathcal{I}}

\newcommand{\cL}{\mathcal{L}}

\DeclareMathOperator{\Id}{Id}
\DeclareMathOperator{\Eff}{Eff}
\DeclareMathOperator{\bigO}{O}
\DeclareMathOperator{\Euler}{Euler}
\DeclareMathOperator{\Nov}{\Xi}

\def\Res{\operatorname{Res}}

\DeclareMathOperator{\Aut}{Aut}
\DeclareMathOperator{\End}{End}

\DeclareMathOperator{\ev}{ev}

\newcommand{\vir}{\text{\rm vir}}

\newcommand{\correlator}[1]{\left \langle #1 \right \rangle}

\def\L{{\mathcal L}}
\def\h{{\hbar}}

\def\str{\operatorname{str}}
\renewcommand\hat{\widehat}
\def\rk{\operatorname{rk}}
\def\ct{\operatorname{ct}}
\def\A{{\mathcal A}}

\begin{document}

\title{Virasoro Constraints for Toric Bundles}

\author{Tom Coates}
\address{Department of Mathematics\\ Imperial College London\\ 180 Queen's Gate\\ London SW7 2AZ\\ UK}
\email{t.coates@imperial.ac.uk}

\author{Alexander Givental}
\address{Department of Mathematics\\ University of California at Berkeley\\ Berkeley, California, 94720\\ USA} 
\email{givental@math.berkeley.edu}

\author{Hsian-Hua Tseng}
\address{Department of Mathematics\\ Ohio State University\\ 100 Math Tower\\ 231 West 18th Avenue\\ Columbus, OH 43210-1174\\ USA}
\email{hhtseng@math.ohio-state.edu}

\date{\today}

\begin{abstract} 
  We show that the Virasoro conjecture in Gromov--Witten theory holds for the
  the total space of a toric bundle $E \to B$ if and only if it holds for
  the base $B$.  The main steps are: (i) we establish a localization formula that expresses
  Gromov--Witten invariants of $E$, equivariant with respect to the fiberwise torus
  action, in terms of genus-zero invariants of the toric fiber and
  all-genus invariants of $B$; and (ii) we pass to the non-equivariant limit in this formula, using Brown's mirror theorem for toric bundles.
 \end{abstract}

\maketitle

{\centering \em To the memory of Bumsig Kim\par}

\section{Introduction} 

Virasoro constraints are differential relations between generating functions for Gromov--Witten invariants of a compact K\"{a}hler manifold $X$. In order to formulate these relations, we begin with a refresher on the structure of Gromov--Witten invariants in genus zero. We assume that the reader is familiar with the basic assumptions and results in Gromov--Witten theory; introductions to the subject, from a compatible point of view, can be found in \cite{Brown_toric, Coates-Givental, Giv_quantization, Giv_Frob}.  

\subsection{Genus-Zero Gromov--Witten Theory} 

Let $H$ denote the classical cohomology algebra of $X$, which we equip with coefficients in the {\em Novikov ring} $\Nov$  of $X$. Following \cite{Giv_quantization, Giv_Frob}, we encode genus-zero Gromov--Witten invariants of $X$ by an {\em overruled Lagrangian cone} $\L_X$ in the {\em symplectic loop space} $(\H, \Omega)$.  Here $\H := H(\!(z^{-1})\!)$ is the $\ZZ_2$-graded module over the Novikov ring consisting of Laurent series in $z^{-1}$ with vector coefficients. The symplectic form on $\H$ is
$$\Omega (f,g) := \Res_{z=0} \big(f(-z), g(z)\big)\,dz$$ where $(\cdot, \cdot)$ is the Poincar\'e pairing on $H$ with values in $\Nov$. The subspaces $\H_{+}:=H[z]$ and $\H_{-}:=z^{-1}H[\![z^{-1}]\!]$ form a Lagrangian polarization of $(\H,\Omega)$, thus identifying it with $T^*\H_{+}$. The Lagrangian cone $\L_X$ is a germ of a Lagrangian section over the point $-1 z \in \H_{+}$, where $1$ is the unit vector in $H$. This section is therefore the graph of differential of a formal function on $\H_{+}$, the genus-zero descendant potential of $X$, although with the domain translated by the \emph{dilaton shift} $\bt \mapsto \bt - 1z$.  The statement that $\L_X$ is {\em overruled} means that each tangent space $T$ to $\L_X$ is a $\Nov [z]$-module, and is in fact tangent to $\L_X$ exactly along $zT$.

\subsection{Grading}
\label{sec:grading}
The fact that $\L_X$ is an overruled cone with the vertex at the origin of $\H$ puts constraints on the genus-zero descendant potential of $X$ which are exactly equivalent to the dilaton equation, string equation, and topological recursion relations \cite{Giv_Frob}. In particular, the string equation can be formulated as invariance of $\L_X$ under the flow of the linear vector field on $\H$ defined by the operator, denoted $l_{-1}$, of multiplication by $z^{-1}$.

To introduce the Virasoro constraints, one needs to invoke one more structure in Gromov--Witten theory: {\em grading}.  Consider the {\em twisted loop group}, that is, the group of operators on $\H$ commuting with $z$ and preserving the symplectic form.  An element $a$ of the Lie algebra of the twisted loop group is an $\End(H)$-valued function of $z$ satisfying $a(-z)^*=-a(z)$, where the asterisk $*$ denotes adjoint with respect to the Poincar\'e pairing on $H$.  The grading condition in Gromov--Witten theory can be formulated as invariance of $\L_X$ with respect to the flow of the linear vector field defined by an operator, denoted $l_0$, of the form: $$l_0 = zd/dz+1/2 + a$$ where $a$ is a suitable element of the Lie algebra of the twisted loop group.  Since $zd/dz+1/2 = \sqrt{z} (d/dz) \sqrt{z}$ is antisymmetric with respect to $\Omega$, the operator $l_0$ is also antisymmetric with respect to $\Omega$. In the case of Gromov--Witten theory of the K\"ahler manifold $X$, $a = \mu+\rho/z$, where $\mu \colon H\to H$ is the {\em Hodge grading operator} (i.e.~the operator of grading in cohomology measured from the middle degree and taking half-integer values), and $\rho$ is the operator of multiplication by $c_1(T_X)$ using the classical cup-product on $H$.

The grading property of $\L_X$ is the consequence of dimensional constraints: Gromov--Witten invariants, being integrals of cohomology classes against the virtual fundamental cycle of a moduli space of stable maps, vanish unless the degree of the class matches the dimension of the cycle. For such constraints to translate into the grading property on $\H$ it is necessary that constants, i.e.~elements of the ground ring, have degree zero. For example, there is no grading property in equivariant Gromov--Witten theory, because generators of the coefficient ring of equivariant cohomology theory have non-zero degrees.  In non-equivariant Gromov--Witten theory, Novikov variables also have non-trivial degrees, $\deg Q^d = \int_d c_1(T_X)$, but nevertheless the grading property holds due to the {\em divisor equation}, which allows one to recast the non-trivial grading of constants into the correction $\rho/z$ to the grading operator $l_0$.  The following description of Virasoro operators, though in a Fourier-dual form,
goes back to E. Getzler's paper~\cite{Getzler:Virasoro}.

\subsection{Virasoro Constraints in Genus Zero} 

We have $[l_0,l_{-1}]=-l_{-1}$.  Consequently the operators:
\begin{align*}
  l_{-1}=z^{-1},&& 
  l_0,&& 
  l_1:=l_0zl_0&&
  l_2:=l_0zl_0zl_0, &&
  \dots,&&
  l_k:=l_0(zl_0)^k,&&
  \dots
\end{align*}
commute as vector fields $x^{k+1}\frac{d}{dx}$ on the line: $[l_m, l_n] = (n-m) l_{m+n}$. The genus-zero Virasoro constraints, which were first proved by X.~Liu--G.~Tian \cite{Liu-Tian}, can be stated and proved as follows.

\begin{proposition}[{see \cite[Theorem~6]{Giv_Frob}}] 
  If the linear vector field on $\H$ defined by $l_0$ is tangent to the overruled Lagrangian cone $\L_X \subset \H$, then the linear vector fields defined by the operators $l_m$,~$m\geq -1$, are all tangent to $\L_X$.  
\end{proposition}

\begin{proof} Let $T$ be the tangent space to $\L_X$ at a point $f$. Then $l_0f \in T$ (by hypothesis), and so $zl_0f \in zT$.  Thus $l_0zl_0f\in T$, which implies that $zl_0zl_0f\in zT$ and hence $l_0zl_0zl_0f\in T$, etc.  \end{proof}

\subsection{Virasoro Constraints in Higher Genus}
\label{sec:Virasoro_higher_genus}

Genus-$g$ Gromov--Witten invariants of $X$ are encoded by the {\em genus-$g$ descendant potential} $\cF^g_X$, which is a formal function on $\H_{+}$ (with coefficients in the Novikov ring) defined near the origin. The totality of Gromov--Witten invariants of $X$ is  encoded by the expression:
\begin{equation}\label{defn:desc_potential_intro}
\cD_X := \exp \left( \sum_{g=0}^{\infty} \h^{g-1} \cF^g_X \right)
\end{equation}
called the {\em total descendant potential} which, after the dilaton shift by  $-1z$, is interpreted as an ``asymptotic element'' of the Fock space associated through quantization with the symplectic loop space $(\H, \Omega)$: see~\cite{Giv_quantization}. The quantization rules by which quadratic Hamiltonians on $\H$ act on  elements of the Fock space are as follows. In a Darboux coordinate system $\{ q^a, p_b\}$ compatible with the polarization $\H=\H_{+}\oplus \H_{-}$, we have: 
\begin{align*}
  \widehat{q^a q^b}:=\textstyle \frac{q^a q^b}{\h}, &&
  \widehat{q^a p_b}:=\textstyle q^a \frac{\partial}{\partial q^b}, &&
  \widehat{p_a p_b}:=\textstyle \h \frac{\partial}{\partial q^a}\frac{\partial}{\partial q^b}
\end{align*}
The linear operators $l_k$ are infinitesimal symplectic transformations and thus correspond to quadratic Hamiltonians. Their quantizations, 
$\hat{l}_m$, satisfy:
\[ [\hat{l}_m, \hat{l_n}] = (n-m)\hat{l}_{m+n}+c_{m,n}\]
where $c_{m,n}=-c_{n,m}$ forms a $2$-cocycle due to the Jacobi identity.  On the Lie algebra of vector fields, any such cocycle is a coboundary, that is, the commutation relations can be restored by adding to the generators $\hat{l}_k$ suitable constants. Namely,  $c_{m,n}=(m-n)c_{m+n,0}/(m+n)$ when $m+n\neq 0$, and $m+n=0$ only when $(m,n)=\pm (-1,1)$. So, the following corrected quantized operators commute  as the classical ones:
\begin{align*}
  L_k:=\hat{l_k}+c_k, &&
  \text{where }
  c_k=
  \begin{cases}
    \frac{c_{k,0}}{k} & \text{if $k \ne 0$} \\
    \frac{c_{1,-1}}{2} & \text{if $k=0$.}
  \end{cases}
\end{align*}
Virasoro constraints, first conjectured by T. Eguchi, K. Hori, M. Jinzenji, C.-S. Xiong, and S. Katz in \cite{EHX, EJX}, can be stated as follows.

\begin{conjecture}[Virasoro Conjecture] 
  If a total descendant potential $\cD$ satisfies the string and grading constraints, i.e.~if $L_{-1}\cD=0$ and $L_0\cD=0$, then it satisfies all higher Virasoro constraints:  $L_k\cD=0$ for $k=1,2,\dots$.
\end{conjecture}

This formulation can be understood as a statement about total descendant potentials $\cD$ of abstract, axiomatically described Gromov--Witten-like theories as introduced by Kontsevich--Manin \cite{Kont-Manin} (see also \cite{Teleman}). When $\cD$ is the total descendant potential $\cD_X$ of a target space $X$, the string equation always holds and the grading constraint $L_0\cD_X=0$ holds in non-equivariant Gromov--Witten theory for dimensional reasons: see e.g.~\cite[Theorem 2.1]{Getzler:Virasoro}, where it is referred to as Hori's equation. In this case, the central constants are:
\[
c_0 = \frac{\chi (X)}{16} + \frac{\str (\mu \mu^*)}{4} 
\]
and $c_k=0$ for $k\neq 0$.
The fact that $\cD_X$ is an eigenfunction of $\hat{l}_0$, the grading operator {\em per se}, with eigenvalue $-c_0$ comes from the anomalous term in the dilaton equation arising from the ``missing'' genus-one degree-zero moduli space of stable maps to $X$. The fact that the eigenvalue, which comes from some Hodge integral over $\overline{\M}_{1,1}\times X$, coincides with the constant $c_0$ dictated by the commutation relations, can be considered non-trivial evidence in favor of the Virasoro Conjecture.

In \cite{Giv_quantization}, Givental described an approach to the Virasoro Conjecture for target spaces $X$ with semisimple quantum cohomology, and proved the Conjecture for toric Fano manifolds.  Subsequently this approach was used to prove the  Conjecture for general toric manifolds~\cite{Iritani}, complete flag manifolds~\cite{Joe-Kim}, Grassmanians~\cite{BCFK}, and all compact K\"ahler manifolds with semisimple quantum cohomology algebras~\cite{Teleman}. The Virasoro Conjecture holds for Calabi--Yau manifolds for dimensional reasons \cite{Getzler:Virasoro}.  It has also been proved for nonsingular curves~\cite{OP}, using an entirely different set of techniques. 

\subsection{Loop Group Covariance} 
\label{sec:covariance}

The Lie algebra of the twisted loop group acts by infinitesimal symplectic transformations on $(\H,\Omega)$, and the central extension of this Lie algebra acts via quantization on elements of the Fock space. Exponentiating, one defines the action of the twisted loop group elements: $\hat{M}=\exp (\widehat{\ln M})$. As a word of warning, we should add that in practice we will need the action of certain elements of completions of the loop group (completions into infinite $z$- or $z^{-1}$-series). Not all such operators can be applied to all elements of the Fock space, nor can such operators be composed arbitrarily. In practice we will use only particular types of quantized loop group elements applied to specific asymptotic elements of the Fock space in such an order that, due to certain nice analytic properties of the functions involved, the application makes sense (sometimes even when the product of infinite matrix series is ill-defined in the loop group itself).  With these warnings out of the way, let us assume that two asymptotical elements of the Fock space, $\cD'$ and $\cD''$, are related by such a loop group transformation: $\cD''=\hat{M}\cD'$. We claim that Virasoro constraints behave covariantly with respect to loop group transformations.

\begin{proposition}[Loop Group Covariance]
  \label{pro:covariance}
  Suppose that $\cD'$ and $\cD''=\hat{M}\cD'$ both satisfy
  the grading constraints $L'_0 \cD' = 0$, $L''_0 \cD'' = 0$ 
  for suitable grading operators $l_0'$ and $l_0''$ on $\H$. Suppose that $M$ respects the grading in the sense that $l_0''=M l_0'M^{-1}$. Then $\cD'$ satisfies Virasoro constraints if and only if $\cD''$ satisfies Virasoro constraints.  
\end{proposition}

\begin{proof} 
  It suffices to show that $L_k''=\hat{M}L_k'\hat{M}^{-1}$ for all $k$. We have that 
  \begin{equation}
    \label{eq:quantize_me}
    l''_k=M l'_k M^{-1}
  \end{equation}
  since $l_k' = l_0' (zl_0')^k$, $l_k'' = l_0'' (zl_0'')^k$, and $M$ commutes with $z$. After quantization the left- and right-hand sides of \eqref{eq:quantize_me} can differ only by central constants $c_k$. Yet $c_0=0$, for otherwise $\cD''$, which is annihilated by both $L_0''$ and $\hat{M}L_0'\hat{M}^{-1}$, would be annihilated by $c_0\neq 0$, implying that $\cD''=\cD'=0$. Notice that for $k\neq 0$, the commutation relation in the Virasoro Lie algebra $[L_0,L_k]=kL_k$ is restored by adding a constant to $L_k$. Since $L_k''$ satisfies the same commutation relations as $\hat{M}L_k'\hat{M}^{-1}$, it follows that $c_k=0$ for all $k$.  \end{proof}

This covariant behavior of Virasoro constraints was the basis for the proof of the Virasoro Conjecture for target spaces with generically semisimple quantum cohomology algebras. Namely, Teleman has proven that the twisted loop group acts transitively on abstract semisimple theories obeying the string equation~\cite{Teleman}. More precisely, the ``upper-triangular'' part of the group, consisting of power series in $z$, acts on all Gromov--Witten-like theories in the Kontsevich--Manin sense, and in particular acts on the corresponding ancestor potentials, whilst the ``lower-triangular'' part, consisting of power series in $z^{-1}$, transforms ancestor into descendant potentials. The combined action transforms (as conjectured in \cite{Giv_quantization}) the descendant potential $\cD_X$ of a semisimple target $X$ into the descendant potential $\cD_{\text{\rm point}}^{\otimes \dim H}$ of a zero-dimensional target. The fact that the descendant potential of the point target satisfies Virasoro constraints is equivalent to the celebrated Witten--Kontsevich theorem \cite{Witten,Kont} relating intersection theory on Deligne--Mumford spaces to the KdV~hierarchy.

\subsection{Toric Bundles} 
\label{sec:toric_bundles}

It is well-known that a compact projective toric manifold $X$ can be obtained by symplectic reduction from a linear space, $X=\CC^N/\!\!/K$, by a subtorus $K := (S^1)^k$ of the maximal torus $T := (S^1)^N$ of diagonal unitary matrices: see e.g.~\cite{Audin,Giv_toric}. We assume without loss of generality that $k=\rk H^2(X)$. Let $B$ be a compact K\"ahler manifold, and let $L_1\oplus \cdots \oplus L_N\to B$ be a rank $N$ complex vector bundle decomposed as a direct sum of line bundles.  The maximal torus $T$ acts fiberwise on this bundle, and one can perform symplectic reduction by the subtorus $K$ fiberwise, thus obtaining a {\em toric bundle} $E\to B$ with fiber~$X$. The group $T$ is abelian, so $E$ still carries a canonical (fiberwise) left action of $T:=(\Cstar)^N$. Let us denote by $E^T$ the fixed point locus of this action. It consists of $n:=\rk H^*(X)$ copies of $B$, which are sections of the bundle $E\to B$. The main result of this paper is:

\begin{theorem} 
  \label{thm:main}
There exists a grading-respecting loop group operator which relates the total descendant potential of a toric bundle space $E$ to that of the fixed point manifold $E^T$: $\cD_E = \hat{M}\cD_{B}^{\otimes n}$. 
\end{theorem}

\noindent The discussion in the preceding two sections then yields:

\begin{corollary}
  The Virasoro Conjecture holds for the total space $E$ if and only if it holds for the base $B$.
\end{corollary}

\begin{corollary} 
The Virasoro Conjecture holds for the total space of a toric bundle over a base $B$ in any of the following cases: 
\begin{enumerate}
\item the quantum cohomology algebra of $B$ is generically semisimple.
\item $B$ is a compact Riemann surface.
\item $B$ is a $K3$ surface.
\item $B$ is a Calabi-Yau manifold of dimension at least $3$.
\end{enumerate}
\end{corollary}

\noindent Note that in case (2) the quantum cohomology of $B$ is  semisimple only if $B=\proj^1$.  In cases (3) and (4) the quantum cohomology of $B$ is never semisimple.

\section{The Proof of Theorem~\ref{thm:main}}

Our proof of Theorem~\ref{thm:main} consists of two steps: (i) relating $T$-equivariant counterparts of the descendant potential of $E$ and $E^T$ by a loop group transformation $M$; and (ii) establishing the existence of a non-equivariant limit of $M$.  The first step relies on fixed point localization in $T$-equivariant Gromov--Witten theory; for the second step we use Brown's relative mirror theorem for toric bundles \cite{Brown_toric}.

\subsection{Descendant-Ancestor Correspondence} 
\label{sec:descendant_ancestor}
To elucidate step (i), we first recall the {\em descendant/ancestor correspondence} \cite[Appendix~2]{Coates-Givental}:
\[ \cD = e^{F(\tau)} \hat{S(\tau)^{-1}}\ \A (\tau).\] Here $\cD$ is the total descendant, and $\A$ the total {\em ancestor} potential of $E$. The latter is defined \cite{Coates-Givental, Giv_quantization} by replacing the $\psi$-classes in the definition of Gromov--Witten invariants with their counterparts pulled back from Deligne--Mumford spaces by the {\em contraction maps} $\ct \colon \overline{\M}_{g,n+m}(E,d)\to \overline{\M}_{g,n}$, and also inserting a primary class $\tau\in H$ at each of the $m$ free marked points (which results in the dependence of $\A$ on the parameter $\tau$). The function $F$ in the exponent is the potential for primary (no descendants) genus-$1$ Gromov--Witten invariants. The operator $S$ is lower-triangular (i.e.~represented by a series in $z^{-1}$) and is uniquely determined by the overruled Lagrangian cone $\L_E$, as follows.  Each tangent space to $\L_E$ is tangent to $\L_E$ at a point of the form $-1z+\tau+\bigO(1/z)$, and thus the tangent spaces depend on a parameter $\tau$. For a tangent space $T_{\tau}$, the $\Nov[z]$-linear projection $\H_{+}\to T_{\tau}$ along $\H_{-}$ determines (and is determined by) the lower-triangular loop group element $S(\tau)^{-1} \colon H \subset \H_{+} \to T_{\tau} \subset \H$.

The operator $S$ can be expressed in terms of Gromov--Witten invariants and thus it is subject to dimensional constraints. This guarantees that conjugation by $S$ respects the grading operators, i.e.~transforms the grading operator $l_0$ for Gromov--Witten theory into the grading operator for the ancestor theory. Thus it remains to find another grading-preserving loop group transformation relating $e^{F(\tau)}\A (\tau)$ with $\cD_{E^T}$.

\subsection{Fixed-Point Localization} 
\label{sec:localization_outline}

In the torus-equivariant version of Gromov--Witten theory, there exists an ``upper-triangular'' element $R (\tau)$ of the loop group which provides the following relationship between suitable ancestor potentials of the target space $E$ and its fixed point locus $E^T$: 
\begin{equation}
  \label{eq:main_localization_formula}
  \A^{\eq}(\tau) = \hat{R(\tau)} \ \prod_{\alpha \in F} \A^{\alpha,\tw}_{B}(u_\alpha(\tau))
\end{equation}
Here $\A^{\eq}$ refers to the total ancestor potential of $E$ in $T$-equivariant Gromov--Witten theory (of which we are reminded by the superscript $eq$), and $\A^{\alpha,\tw}_{B}$ is a similar ancestor potential of one component $E^\alpha =  B$ of the fixed point locus $E^T$. The superscript $tw$ indicates that we are dealing here not with the Gromov--Witten theory of $E^{\alpha}=B$ {\em per se}, but with the Gromov--Witten theory of the normal bundle of $E^{\alpha}$ in $E$; this is the \emph{local} theory\footnote{This is necessarily an equivariant theory, as the target space is non-compact.} of $E^\alpha$, or in other words the \emph{twisted} theory of $B$. The product over the fixed point set $F=X^T$ means that each function depends on its own group of variables according to the decomposition $H_T:=H_T^*(E;\Nov) = \oplus_{\alpha \in F} H_T^*(E^{\alpha};\Nov)$ of equivariant cohomology induced by the embedding $E^T\subset E$ and localization. Let us write $\tau = \oplus_{\alpha} \tau_{\alpha}$ using the same decomposition.  The quantities $u_\alpha(\tau)$ in \eqref{eq:main_localization_formula} are certain block-canonical coordinates on $H_T$, defined in \S\ref{sec:genus_zero} below, which have the property that $u_\alpha(\tau) \equiv \tau_\alpha$ modulo Novikov variables. The operator $R(\tau)$ here is a power series in $z$, and it does not have a non-equivariant limit.  The existence of the operator $R(\tau)$ and the validity of formula \eqref{eq:main_localization_formula} are established in \S\ref{sec:localization} below.

Next, the ancestor potential of the normal bundle of $E^\alpha$ in $E$ is related to the corresponding descendant potential by the equivariant version of the descendant/ancestor correspondence: 
\begin{align*}
  \A^{\alpha,\tw}_{B} (u_\alpha(\tau)) = e^{-F^{\alpha,\tw}_{B}(u_\alpha(\tau))} \hat{S_{B}^{\alpha,\tw}(u_\alpha(\tau))} \cD_{B}^{\alpha,\tw}
  &&
  \text{for each $\alpha \in F$.}
\end{align*}
Finally, according to the Quantum Riemann--Roch Theorem\footnote{To simplify discussion, we omit a constant factor in the quantum Riemann-Roch formula.} \cite{Coates-Givental}:
\begin{equation}
  \label{eq:QRR}
  \cD_{B}^{\alpha,\tw} = \widehat{\Gamma_{\alpha}^{-1}} \cD_{B}^{\eq}
\end{equation}
That is, the twisted equivariant descendant potential is obtained from the
untwisted equivariant  descendant potential of the fixed point manifold 
(i.e.~$|F|$ copies of the base $B$ of the toric bundle in our case) by quantized
loop group transformations. The operators involved have the form
\[ 
\Gamma_{\alpha}^{-1} =\exp \left(\sum_{m\geq 0} \rho_m^{\alpha} z^{2m-1}\right)
\] 
where $\rho_m^{\alpha}$ are certain operators of multiplication in the classical equivariant cohomology of $E^{\alpha}$; see~\cite{Coates-Givental} for the precise definition.

Composing the above transformations, we obtain 
\begin{align*}
e^{F^{\eq}(\tau)}\A^{\eq}(\tau) =\hat{M(\tau)} \prod_{\alpha \in F} \cD^{\eq}_B
&&
\text{where}
&&
\hat{M(\tau)}=e^{F^{\eq}(\tau)-\sum_{\alpha} F^{\alpha,\tw}_{B}(u_\alpha(\tau))} \hat{R}(\tau) 
\left(\bigoplus_{\alpha\in F} \hat{S_{B}^{\alpha,\tw}}\big(u_{\alpha}(\tau)\big) \hat{\Gamma_{\alpha}^{-1}}\right)
\end{align*}
Both $e^{F^{\eq}(\tau)}\A^{\eq}(\tau)$ and $\prod_{\alpha\in F} \cD_{B}^{\eq}$ have non-equivariant limits, 
but some ingredients of the operator $M(\tau)$ do not.  Nonetheless, we prove:

\begin{claim*}
  The operator $M(\tau)$ has a well-defined non-equivariant limit, which is grading-preserving.
\end{claim*}
 
\noindent This implies Theorem~\ref{thm:main}.

\subsection{Why Does The Limit Exist?}
To understand why the non-equivariant limit of $M(\tau)$ exists, consider the equivariant descendant/ancestor relation:
\[
e^{F^{\eq}(\tau)}\A^{\eq}(\tau) = \hat{S^{\eq}}(\tau) \cD^{\eq}
\]
The descendant potential here does not depend on the parameter $\tau \in H_T$.  The upper-triangular loop group element $S^{\eq}(\tau)$ is a fundamental solution to the (equivariant) quantum differential equations\footnote{This is the equivariant version of the Dubrovin connection.}  for $E$:
\begin{align*}
  z \partial_v S (\tau) & = v \bullet_{\tau} S (\tau)
  &
  v \in H_T \\
  \intertext{where $\bullet_{\tau}$ is the equivariant big quantum product with parameter $\tau \in H_T$. 
    Therefore the function $e^{F^{\eq}}\A^{\eq}$ depends on $\tau$ in the same way, i.e.~satisfies:}
  \partial_v \Big(e^{F^{\eq}(\tau)}\A^{\eq}(\tau)\Big) &= \hat{\left(\frac{v\bullet_{\tau}}{z}\right)} \Big(e^{F^{\eq}(\tau)}\A^{\eq}(\tau)\Big)
  & 
  v \in H_T. \\
  \intertext{On the other hand, $e^{F^{\eq}}\A^{\eq}=\hat{M(\tau)}\prod_{\alpha\in F} \cD_{B}^{\eq}$ and $\prod_{\alpha\in F} \cD_{B}^{\eq}$ does not depend on $\tau$. Therefore the dependence of $M$ on $\tau$ is governed by the same connection:}
  z\partial_v M (\tau) & = v\bullet_{\tau} M(\tau)
  & v \in H_T.
\end{align*}

The main result of the paper \cite{Brown_toric} about toric bundles provides, informally speaking, a fundamental solution to this connection for the total space $E$ of a toric bundle assuming that the fundamental solution for the base $B$ is known. The solution is given in the form of an oscillating integral---the ``equivariant mirror'' to the toric fiber $X$. The proof of the above Claim is based on the identification of $M$ with another form of this solution, obtained by replacing the oscillating integrals with their {\em stationary phase asymptotics}. In the non-equivariant limit of the oscillating integrals, the stationary phase asymptotics tend to a well-defined limit because the critical points of the phase function remain well-defined and non-degenerate in the limit -- this is equivalent to the semisimplicity of the non-equivariant quantum cohomology algebra of the toric fiber~\cite{Iritani}. We explain this in detail in \S\ref{sec:non-equivariant_limit} below.  In the next section, we discuss localization in $T$-equivariant Gromov--Witten theory of toric bundles.

\section{Fixed-Point Localization}
\label{sec:localization}

\subsection{The $T$-Action on $E$}
\label{sec:more_toric_bundles}

Let $\pi \colon E \to B$ be a toric bundle with fiber $X$, constructed as in \S\ref{sec:toric_bundles}.  Let $\mathfrak{k}$ and $\mathfrak{t}$ denote the Lie algebras of $K = (S^1)^k$ and $T = (S^1)^N$ respectively.  Our assumption that $\rk H^2(X)=k$ implies\footnote{See \cite{Audin} or \cite{Giv_toric} for details.} that there is a canonical isomorphism $\mathfrak{k}^\vee \cong H^2(X;\RR)$, and so the symplectic form on the toric fiber $X$ determines a point $\omega \in \mathfrak{k}^\vee$.  The embedding $K \to T$ determines and is determined by a linear map $\mathfrak{k} \to \mathfrak{t}$; this map is given by a $k \times N$ matrix with integer entries $(m_{ij})_{1 \leq i \leq k,1 \leq j \leq N}$.  The columns of $(m_{ij})$ determine elements $D_1,\ldots,D_N \in H^2(X;\RR)$, the toric divisors on $X$.  

As discussed above, the total space $E$ of the toric bundle carries a fiberwise left action of the big torus~$T$.  The $T$-fixed set $E^T$ consists of $n = \rk H^*(X)$ copies of $B$, which are the images of sections of $\pi$.  The $T$-fixed points on $X$, and also the $T$-fixed strata on $E$, are indexed by subsets $\alpha \subset \{1,2,\ldots,N\}$ of size $k$ such that $\omega$ lies in the cone spanned by $\{D_i : i \in \alpha\}$.  Denote the set of all such subsets $\alpha$ by $F$.  Given such a subset $\alpha \in F$, we write $x^\alpha$ for the corresponding $T$-fixed point in $X$ and $E^\alpha$ for the corresponding $T$-fixed stratum in $E$.  One-dimensional $T$-orbits on $E$ are non-isolated (unless $B$ is a point), and components of this space are indexed by one-dimensional $T$-orbits on $X$.  A one-dimensional $T$-orbit in $X$ that connects $x^\alpha$ to $x^\beta$ corresponds to the component of the space of one-dimensional orbits in $E$ consisting of orbits that connect $E^\alpha$ to $E^\beta$; this component is again a copy of the base~$B$.  We write $\beta \to \alpha$ if there is a one-dimensional $T$-orbit in $X$ from $x^\alpha$ to $x^\beta$.  
For each $\beta \in F$ such that $\beta \to \alpha$ there is a line bundle over $E^\alpha \cong B$ formed by tangent lines (at $E^\alpha)$ to closures of the $1$-dimensional orbits connecting $E^\alpha$ to $E^\beta$; we denote the first Chern class of this line bundle by $\chi_{\alpha\beta}$.

\subsection{The $T$-Action on the Moduli Space of Stable Maps}
\label{sec:define_Gamma}

In this section and the next, we describe the technique of fixed point localization on moduli spaces of stable maps. This material, which is well-known, is included for completeness, but we will need very little of it in what follows.  In what comes afterwards, torus-invariant stable maps will be chopped into ``macroscopic'' pieces. What one needs to digest from the ``microscopic'' description given here is that the moduli spaces of torus-invariant stable maps factor acording to the pieces, that integrals over the factors are naturally assembled into appropriate Gromov--Witten invariants, and that the only integrands which do not behave multiplicatively with respect to the pieces are the ``smoothing factors'' defined below.  The impatient reader should therefore skip straight to \S\ref{sec:genus_zero}, pausing only to examine the explicit forms of the smoothing factors, which are given just after equation \eqref{eq:Euler}.

Let $E_{g,n,d}$ denote the moduli space of degree-$d$ stable maps to $E$ from curves of genus~$g$ with $n$~marked points; here $g$ and $n$ are non-negative integers and $d \in H_2(E;\ZZ)$.  The action of $T$ on $E$ induces an action of $T$ on $E_{g,n,d}$.  
A stable map $f \colon C \to E_{g,n,d}$ representing a $T$-fixed point in $E_{g,n,d}$ necessarily has $T_{\CC}$-invariant image, which lies therefore in the union of $0$- and $1$-dimensional $T_{\CC}$-orbits in $E$.  More precisely, some rational irreducible components of $C$ are mapped onto $1$-dimensional $T_{\CC}$-orbit closures as multiple covers $z \mapsto z^k$ (in obvious coordinates).  We call them \emph{legs} of multiplicity $k$.  After removing the legs, the domain curve $C$ falls into  connected components $C_v$, and the restrictions $f |_{C_v}$ are stable\footnote{Marked points on $C_v$ here are marked points from $C$ and the attaching points of legs.} maps to $E^T$ which we call \emph{stable pieces}.  Note that each leg maps $z=0$ and $z=\infty$ to the $T$-fixed locus $E^T \subset E$, and each of $z=0$ and $z=\infty$ is one of:
\begin{enumerate}
\item a node, connecting to a stable piece;
\item a node, connecting to another leg;
\item a marked point; 
\item an unmarked smooth point.
\end{enumerate}
We associate to possibilities (1)--(4) vertices of $\Gamma$, of types 1--4 respectively, thereby ensuring that each leg connects precisely two vertices.  The combinatorial structure of a $T$-fixed stable map can be represented by a decorated graph $\Gamma$, with vertices as above and an edge for each leg.  The edge $e$ of $\Gamma$ is decorated by the multiplicity $k_e$ of the corresponding leg.  The vertex $v$ of $\Gamma$ is decorated with $(\alpha_v,g_v,n_v,d_v)$ where $\alpha_v \in F$ is the component of the $T$-fixed set determined by $v$, and $(g_v,n_v,d_v)$ record the genus, number of marked points, and degree of the stable map $f|_{C_v}$ for vertices of type~$1$, and:
\[
(g_v,n_v,d_v) = 
\begin{cases}
  (0,2,0) & \text{for type 2} \\
  (0,1,0) & \text{for type 3} \\
  (0,0,0) & \text{for type 4.} 
\end{cases}
\]
As a $T$-fixed stable map varies continuously, the combinatorial type $\Gamma$ does not change.  Each connected component\footnote{There are many such components corresponding to each decorated graph $\Gamma$, which differ by the numbering of the marked points assigned to each vertex.} is isomorphic to a substack of:
\begin{equation}
  \label{eq:T_fixed_component}
  \left(\prod_{\text{vertices $v$ of $\Gamma$}}
    B_{g_v,n_v,d_v}
  \right) \Big/ \left(\Aut(\Gamma) \times \prod_{\text{edges $e$ of $\Gamma$}} \ZZ/k_e\ZZ \right)
\end{equation}
Here $B_{g_v,n_v,d_v}$ is the moduli space of stable maps to $B$, and the `missing' moduli spaces $B_{0,2,0}$, $B_{0,1,0}$, and $B_{0,0,0}$ are taken to be copies of $B$.  The substack is defined by insisting that, for each edge $e$ between vertices $v$,~$w$ of $\Gamma$, the evaluation maps
\begin{align*}
  \ev_e \colon B_{g_v,n_v,d_v} \to B && \text{and} && \ev_e \colon B_{g_{w},n_{w},d_{w}} \to B
\end{align*}
determined by the edge $e$ have the same image.  These constraints reflect the fact that the one-dimensional orbit in $E$ determined by the edge $e$ runs along a fiber of the toric bundle $E \to B$.

\subsection{Virtual Localization}  We will compute Gromov--Witten invariants of $E$ by virtual localization.  The Localization Theorem in equivariant cohomology states that, given a holomorphic action of a complex torus $T_{\CC}$ on a compact complex manifold $M$ and $\omega \in H^\bullet_T(M)$, we have
\[
\int_{[M]} \omega = \int_{[M^T]} \frac{i^* \omega}{\Euler(N_{M^T})}
\]
where $i \colon M^T \to M$ is the inclusion of the $T$-fixed submanifold, $N_{M^T}$ is the normal bundle to $M^T$ in $M$, $\Euler$ denotes the $T$-equivariant Euler class, and the integrals denote the evaluation of a $T$-equivariant cohomology class against the $T$-equivariant fundamental cycle.  According to Graber--Pandharipande~\cite{Graber-Pandharipande}, the same formula holds when $M$ is the moduli space of stable maps to a smooth projective $T$-variety; $[M]$ and $[M^T]$ denote $T$-equivariant virtual fundamental classes~\cite{Li--Tian,Behrend--Fantechi}; and $N_{M^T}$ denotes the virtual normal bundle to $M^T$, that is, the moving part of the virtual tangent bundle to $M$ restricted to $M^T$.  To apply this virtual localization formula, we need to describe $[M^T]$ and $\Euler(N_{M^T})$.  The $T$-fixed components~\eqref{eq:T_fixed_component} come equipped with virtual fundamental classes from the Gromov--Witten theory of $B$, and these give the virtual fundamental class $[M^T]$.  We next describe $\Euler(N_{M^T})$.

Consider a connected component \eqref{eq:T_fixed_component} consisting of stable maps with combinatorial type $\Gamma$.  The analysis of the virtual tangent bundle to the moduli space of stable maps in~\cite{Graber-Pandharipande,Kontsevich:enumeration} shows that $\Euler(N_{M^T})$ takes the form:
\begin{equation}
  \label{eq:Euler}
  \Euler(N_{M^T}) = C_{\text{\rm smoothing}} \, C_{\text{\rm vertices}} \, C_{\text{\rm edges}} 
\end{equation}
Here the factor $C_{\text{\rm smoothing}}$ records the contribution from deformations which smooth nodes in the $T$-fixed stable maps of type $\Gamma$.  To each type-$1$ flag, that is, each pair $(v,e)$ where $v$ is a type-$1$ vertex and $e$ is an edge of $\Gamma$ incident to $v$, there corresponds a $1$-dimensional smoothing mode which contributes
\[
\frac{\ev_e^* \chi_{\alpha,\beta}}{k_e} - \psi_e
\]
to $C_{\text{\rm smoothing}}$.  Here $\psi_e$ is the cotangent line class (on $B_{g_v,n_v, d_v}$) at the marked point determined by $e$.  To each type-$2$ flag, there corresponds a $1$-dimensional smoothing mode which contributes
\[
\frac{\chi_{\alpha,\beta_1}}{k_1} + \frac{\chi_{\alpha,\beta_2}}{k_2} 
\]
to $C_{\text{\rm smoothing}}$, where the legs incident to the type-2 vertex connect $E^\alpha$ to $E^{\beta_1}$ and~$E^{\beta_2}$, with multiplicities $k_1$ and $k_2$ respectively.  Flags of types~3 and~4 do not contribute to $C_{\text{\rm smoothing}}$.

The factor $C_{\text{\rm vertices}}$ in \eqref{eq:Euler} is a product over the vertices of $\Gamma$.  A type-$1$ vertex contributes the $T$-equivariant Euler class of the virtual vector bundle $N^{\alpha_v}_{g_v,n_v,d_v}$ over $B_{g_v,n_v,d_v}$ with fiber at a stable map $f \colon C \to B$ given by\footnote{It is shown in Appendix~1 in~\cite{Coates-Givental} that $N^{\alpha_v}_{g_v,n_v,d_v}$ is a well-defined element of $K^0_T(B_{g_v,n_v,d_v})$.}:
\[
H^0(C,f^* N^{\alpha_v}) \ominus H^1(C,f^* N^{\alpha_v})
\]
Here $N^{\alpha_v} \to B$ is the normal bundle to $E^{\alpha_v} \cong B$ in $E$.  Vertices of type $2$,~$3$, and~$4$ contribute, respectively:
\begin{align*}
  \Euler_T(N^{\alpha_v}), &&  \Euler_T(N^{\alpha_v}), &&  \text{and} && \frac{\Euler_T(N^{\alpha_v})}{\chi_{\alpha_v,\beta}/k_e},
\end{align*}
where the leg terminating at a type-4 vertex has multiplicity $k_e$ and connects $E^{\alpha_v}$ to $E^\beta$. 

 The factor $C_{\text{\rm edges}}$ in \eqref{eq:Euler} contains all other contributions from the moving part of the virtual tangent bundle.  We will not need the explicit formula in what follows, but include it here for completeness.  
The factor $C_{\text{\rm edges}}$ is a product over edges of $\Gamma$.  An edge of multiplicity $k$ connecting fixed points indexed by $\alpha$,~$\beta \in F$ contributes
\[
\prod_{j=1}^N 
\frac{
  \prod_{m=-\infty}^0 \big( U_j(\alpha) + m \frac{\chi_{\alpha, \beta}}{k} \big)
}{
  \prod_{m=-\infty}^{D_j \cdot d - 1} \big( U_j(\alpha) + m \frac{\chi_{\alpha, \beta}}{k} \big)
}
\]
where $U_j(\alpha) \in H^2_T(E)$ is the $T$-equivariant class of the $j$th toric divisor, restricted to the fixed locus $E^\alpha$, and $D_j \cdot d$ is the intersection index of that divisor with the degree of the corresponding multiply-covered $1$-dimensional $T$-orbit. See~\cite{Giv_toric} and \cite{Brown_toric} for details.

\subsection{Virtual Localization in Genus Zero}
\label{sec:genus_zero}
Consider the following generating functions for genus-zero Gromov--Witten invariants:
\begin{itemize}
\item the $J$-function $J(\tau,z) \in H^\bullet_T(E)$ defined by $J(\tau,z) = \sum_\mu J(\tau,z)^\mu \phi_\mu$ where:
\begin{equation}\label{def:J}
J(\tau,z)^\mu = 
(1,\phi^\mu) z + \tau^\mu + 
\sum_{d \in H_2(E)} \sum_{n=0}^\infty 
\frac{Q^d}{n!}
\correlator{\tau,\tau,\ldots,\tau,\frac{\phi^\mu}{z-\psi_{n+1}}}_{0,n+1,d}^E
\end{equation}
\item the fundamental solution $S(\tau,z) \colon H^\bullet_T(E) \to H^\bullet_T(E)$ defined by:
  \begin{equation}
    \label{eq:S}
    {S(\tau,z)^\mu}_\nu
    =
    (\phi^\mu,\phi_\nu) + 
    \sum_{d \in H_2(E)} \sum_{n=0}^\infty 
    \frac{Q^d}{n!}
    \correlator{\phi^\mu,\tau,\ldots,\tau,\frac{\phi_\nu}{z-\psi_{n+2}}}_{0,n+2,d}^E
  \end{equation}
\item the bilinear form $V(\tau,w,z)$ on $H^\bullet_T(E)$ defined by:
  \begin{equation}
    \label{eq:V}
    V(\tau,w,z)_{\mu \nu} = 
    \frac{(\phi_\mu,\phi_\nu)}{w+z} + \sum_{d \in H_2(E)} \sum_{n=0}^\infty 
    \frac{Q^d}{n!}
    \correlator{\frac{\phi_\mu}{w-\psi_1},\tau,\ldots,\tau,\frac{\phi_\nu}{z-\psi_{n+2}}}_{0,n+2,d}^E
  \end{equation}
\end{itemize}
Here $\phi_1,\ldots,\phi_{\rk H_T^\bullet(E)}$ and $\phi^1,\ldots,\phi^{\rk H_T^\bullet(E)}$ are bases for $H^\bullet_T(E)$ that are dual with respect to the $T$-equivariant Poincar\'e pairing on $E$, endomorphisms $M$ of $H^\bullet_T(E)$ have matrix coefficients such that $M(\phi_\nu) = \sum_\mu {M^\mu}_\nu \phi_\mu$, and bilinear forms $V$ on $H^\bullet_T(E)$ have matrix coefficients such that $V(\phi_\mu,\phi_\nu) = V_{\mu \nu}$.

 The fundamental solution $S(\tau,z)$ satisfies the $T$-equivariant quantum differential equations:
\begin{align}\label{eqn:QDE}
  z \partial_v S (\tau,z) = v \bullet_{\tau} S (\tau,z)
  &&
  v \in H_T^\bullet(E) 
\end{align}
together with the normalization condition $S(\tau,z) = \Id + \bigO(z^{-1})$.  Standard results in Gromov--Witten theory imply that:
\begin{align}
  \label{eq:JSV}
  J(\tau,z) = z S(\tau,z)^* 1 && \text{and} && V(\tau,w,z) = \frac{S(\tau,w)^* S(\tau,z)}{w+z}
\end{align}
where $S(\tau,z)^*$ denotes the adjoint of $S(\tau,z)$ with respect to the equivariant Poincar\'e pairing on $H_T^\bullet(E)$, and we identify $H^\bullet_T(E)$ with its dual space via the equivariant Poincar\'e pairing (thus equating the bilinear form $V$ with an operator).  The analogous statements hold in the Gromov--Witten theory of $B$ twisted by the normal bundle $N^\alpha$ to the $T$-fixed locus $E^\alpha \cong B$.  

We begin by processing $S(\tau,z)$ by fixed-point localization.  Henceforth we work over the field of fractions of the coefficient ring $H^\bullet_T(pt) = H^\bullet(BT)$ of $T$-equivariant cohomology theory, and insist that our basis $\phi_1,\ldots,\phi_{\rk H_T^\bullet(E)}$ for $H^\bullet_T(E)$ is compatible with the fixed-point localization isomorphism
\[
H^\bullet_T(E) \longrightarrow \bigoplus_{\alpha \in F} H^\bullet_T(E^\alpha)
\]
in the sense that each $\phi_i$ restricts to zero on all except one of the $T$-fixed loci, which we denote by $E^{\alpha_i} \subset E$.

\begin{proposition} \label{pro:R}
  The fundamental solution $S(\tau,z)$ can be factorized as the product
  \begin{equation}
    \label{eq:block_product}
    S(\tau,z) = R(\tau,z) \, S_{\block}(\tau,z)
  \end{equation}
  where $R$ has no pole at $z=0$, and $S_{\block}(\tau,z)$ is the block-diagonal transformation
    \[
    S_{\block}(\tau,z) = \bigoplus_{\alpha \in F} S^{\alpha, \tw}\big(u_\alpha(\tau),z\big).
    \]
    Here $S^{\alpha, \tw}(u,z)$ is the fundamental solution in the  Gromov--Witten theory of $B$ twisted by the normal bundle $N^\alpha$ to the $T$-fixed locus $E^\alpha \cong B$, and  $\tau \mapsto \oplus_\alpha u_\alpha(\tau)$ is a certain non-linear change of coordinates with $u_\alpha(\tau) \in H^\bullet(E^\alpha) \subset H^\bullet(E)$.
\end{proposition}

\begin{remark}
  Comparing the definition of $u_\alpha(\tau)$ given in the proof below with~\cite{Giv_semisimple,Giv_quantization}, one sees that the $u_\alpha(\tau)$ can be regarded as `block-canonical coordinates' of $\tau$.
\end{remark}

\begin{remark}
  Since both multiplication by $S(\tau,z)$ and multiplication by $S_{\block}(\tau,z)$ define linear symplectomorphisms $\cH \to \cH$, so does multiplication by $R(\tau,z)$.  That is, $R(\tau,-z)^* R(\tau,z) = \Id$.
\end{remark}

\begin{proof}[Proof of Proposition~\ref{pro:R}]
  The $z$-dependence in ${S(\tau,z)^\mu}_\nu$ arises only from the input $\frac{\phi_\nu}{z-\psi_{n+2}}$ in \eqref{eq:S} at the last marked point.  Let $E^\alpha$ denote the fixed-point component on which $\phi_\nu$ is supported.  By fixed-point localization, we see that ${S(\tau,z)^\mu}_\nu$ is a sum of contributions from fixed-point components $E^\Gamma_{0,n+2,d}$ where the graphs $\Gamma$ can be described as follows.  A typical $\Gamma$ has a distinguished vertex, called the \emph{head vertex}, that carries the $(n+2)$nd marked point with insertion $\frac{\phi_\nu}{z-\psi_{n+2}}$.  The head vertex is a stable map to $E^\alpha$; it is incident to $m$ trees (the \emph{ends}) that do not carry the first marked point, and also to one distinguished tree (the \emph{tail}) that carries the first marked point.  Thus ${S(\tau,z)^\mu}_\nu$ has the form:
  \begin{equation}
    \label{eq:S_localized}
    \begin{aligned}
       {\delta^\mu}_{\nu} + \sum_{\beta: \beta \to \alpha} \sum_{k=1}^\infty  
        \frac{\Big(T^{k}_{\beta,\alpha}(\phi^\mu),\phi_\nu\Big)}{\frac{\chi_{\alpha \beta}}{k} + z} + \qquad \qquad \qquad \qquad \qquad \qquad \qquad \qquad \qquad \qquad \qquad \qquad  \\
        \qquad \qquad \sum_d \sum_{m=0}^\infty \frac{Q^d}{m!}
        \correlator{ \phi^\mu + \sum_{\beta: \beta \to \alpha} \sum_{k=1}^\infty \frac{T^{k}_{\beta,\alpha}(\phi^\mu)}{\frac{\chi_{\alpha \beta}}{k} - \psi_1}, \epsilon_\alpha(\psi_2),\ldots,\epsilon_\alpha(\psi_{m+1}),\frac{\phi_\nu}{z-\psi_{m+2}} }^{B,\tw,\alpha}_{0,m+2,d} 
    \end{aligned}
  \end{equation}
  where the correlators represent integration over the moduli space of stable maps given by the head vertex, $\epsilon_\alpha$ represents the contribution of all possible ends, the linear map $T^{k}_{\beta,\alpha}\colon H_T^\bullet(E) \to H_T^\bullet(E^\alpha;\Nov)$ records the contribution of all possible tails that approach $E^\alpha$ along an edge from the $T$-fixed component $E^\beta \subset E$ with multiplicity $k$, and $\chi_{\alpha \beta}$ is the cohomology class on $E^\alpha \cong B$ defined in Section~\ref{sec:more_toric_bundles}.  The sum is over degrees $d \in H_2(B;\ZZ)$, and $Q^d$ represents, in the Novikov ring of $E$, the degree of these curves in $E^\alpha \cong B$.  Here $\epsilon_\alpha = \tau$ modulo Novikov variables, as we include in our definition of $\epsilon_\alpha$ the degenerate case where the end consists just of a single marked point attached to the head vertex.  Except ${\delta^\mu}_{\nu}$, the terms on the first line of  \eqref{eq:S_localized}  arise from those exceptional graphs $\Gamma$ where the head vertex is unstable as a map to $E^\alpha \cong B$, that is, where the head vertex is a type-3 vertex in the sense of \S\ref{sec:define_Gamma}.

In fact, \eqref{eq:S_localized} can be written as
\begin{equation}
  \label{eq:S_localized_processed}
     {\delta^\mu}_{\nu} + \sum_{\beta: \beta \to \alpha} \sum_{k=1}^\infty
    \frac{1}{z + \frac{\chi_{\alpha \beta}}{k}}
    \Bigg(
      S^{\alpha,\tw}\big(u_\alpha(\tau),{\textstyle \frac{\chi_{\alpha \beta}}{k}}\big)
      T^k_{\beta,\alpha}(\phi^\mu), \,
      S^{\alpha,\tw}\big(u_\alpha(\tau),z\big) \phi_\nu
    \Bigg)^{\alpha,\tw}
\end{equation}
where $(\cdot,\cdot)^{\alpha,\tw}$ is the twisted Poincar\'e pairing on $E^\alpha$.  This is a consequence of a general result about the structure of genus-zero Gromov--Witten invariants, applied to the Gromov--Witten theory\footnote{By the Gromov--Witten theory of $E^\alpha$ here we mean the Gromov--Witten theory of $B$ twisted by the normal bundle $N^\alpha$ to $E^\alpha \cong B$ and the equivariant inverse Euler class.} of $E^\alpha$.  Recall that the tangent space to the Lagrangian cone $\cL^\tw_{E^\alpha}$ at the point $\bt \in \H_+$ is the graph of the differential of the quadratic form on $\cH_+$ given by
\[
\epsilon \longmapsto \Res_{z=0} \Res_{w=0} \VV(\bt,{-w},{-z})\big(\epsilon(z),\epsilon(w)\big) \, dz \, dw
\]
where
\[
\VV(\bt,w,z) \big(\epsilon_1,\epsilon_2 \big) =
\frac{\big(\epsilon_1(z),\epsilon_2(w)\big)^{\alpha,\tw}}{z+w} +
\sum_d \sum_{n=0}^\infty
\frac{Q^d}{n!}
\correlator{ \frac{\epsilon_1(z)}{z-\psi_1},\bt(\psi_2),\bt(\psi_3),\ldots,\bt(\psi_{n+1}),
    \frac{\epsilon_2(w)}{w-\psi_{n+2}} }^{B,\alpha, \tw}_{0,n+2,d} 
\]
But this tangent space has the form $S^{\alpha,\tw}(u_\alpha,z)^{-1} \cH_+$ for some point $u_\alpha \in H^\bullet_T(E^\alpha)$ that is determined by $\bt$.  In fact, 
\begin{equation}
  \label{eq:what_is_u_alpha}
  u_\alpha = \sum_{\epsilon} \sum_d \sum_{n=0}^\infty \frac{Q^d}{n!}
  \correlator{1,\phi^\epsilon,\bt,\ldots,\bt}^{\alpha,\tw}_{0,n+2,d} \phi_\epsilon
\end{equation}
where the sum is over $\epsilon$ such that $\phi_\epsilon$ is supported on $E^\alpha$; this is the Dijkgraaf--Witten formula~\cite[Equation (2.2)]{Dijkgraaf--Witten} (in the context of Gromov-Witten theory, see also the proof of \cite[Theorem 1]{Giv_Frob} and \cite[Section 4.4]{Getzler:jet}).  As a result:
\[
\VV(\bt,w,z) \big(\epsilon_1,\epsilon_2 \big) =\VV(u_\alpha,w,z) \big(\epsilon_1,\epsilon_2 \big) =
\frac{1}{z+w} \Bigg( S^{\alpha,\tw}(u_\alpha,z) \epsilon_1(z),\,
S^{\alpha,\tw}(u_\alpha,w) \epsilon_1(w) \Bigg)^{\alpha,\tw}
\]
where we used \eqref{eq:JSV}.  Applying this to \eqref{eq:S_localized} yields \eqref{eq:S_localized_processed}, where $u_\alpha(\tau)$ is given by \eqref{eq:what_is_u_alpha} with $\bt$ replaced by the contribution $\epsilon_\alpha$ from all possible ends.

Setting
\begin{equation}
  \label{eq:R_explicit}
  R(\tau,z) = \Id + \sum_{\alpha \in F} \sum_{\beta: \beta \to \alpha} \sum_{k=1}^\infty  
  \frac{1}{z + \frac{\chi_{\alpha \beta}}{k}}
  \Big(T^k_{\beta,\alpha}\Big)^* 
  S^{\alpha,\tw}\big(u_\alpha(\tau),{\textstyle \frac{\chi_{\alpha \beta}}{k}}\big)^*
\end{equation}
yields the result.  This expression has no pole at $z=0$.
\end{proof}

\begin{remark}
  The end contribution $\epsilon_\alpha$ occurring in the proof of Proposition~\ref{pro:R} can be identified in terms of fixed-point localization for the $J$-function:
  \[
  J(\tau,{-z})\big|_{E^\alpha} = 
  {-1z} + \epsilon_\alpha(z) + \sum_d \sum_{m=0}^\infty \sum_\nu \frac{Q^d}{m!}
        \correlator{\epsilon_\alpha(\psi_1),\ldots,\epsilon_\alpha(\psi_m),\frac{\phi^\nu}{z-\psi_{m+1}} }^{B,\alpha, \tw}_{0,m+1,d} \phi_\nu \big|_{E^\alpha}
  \]
  where we processed the virtual localization formulas exactly as in the proof of Proposition~\ref{pro:R}.  Since $\psi_{n+1}$ is nilpotent, the correlator terms have poles (and no regular part) at $z=0$.  At the same time, the summand $\epsilon_\alpha(z)$ has no pole at $z=0$, as it is equal to $\tau|_{E^\alpha}$ plus a sum of terms of the form $\frac{c}{-z + \chi_{\alpha \beta}/k}$ where $c$ is independent of $z$.  Thus 
  \begin{equation}
    \label{eq:end_from_J}
    \epsilon_\alpha(z) = z + \Big[ J(\tau,{-z})\big|_{E^\alpha} \Big]_+ = z - \Big[ z S(\tau,{-z})^* 1\big|_{E^\alpha} \Big]_+
  \end{equation}
  where $[ \cdot ]_+$ denotes taking the power series part of the Laurent expansion at $z=0$.  For the last equality here we used \eqref{eq:JSV}.
\end{remark}

\subsection{Virtual Localization for the Ancestor Potential}  We will use virtual localization to express generating functions for Gromov--Witten invariants of $E$ in terms of generating functions for twisted Gromov--Witten invariants of $E^T$.  The $T$-equivariant \emph{total ancestor potential} of $E$ is
\begin{equation}
  \label{eq:ancestor_potential}
  \cA(\tau;\bt) = \exp\left(\sum_{g=0}^\infty \hbar^{g-1} \bar{\cF}_g(\tau;\bt) \right)
\end{equation}
where $\bar{\cF}_g$ is the $T$-equivariant genus-$g$ ancestor potential\footnote{Integrals over non-existing moduli spaces are defined to be $0$.}:
\[
\bar{\cF}_g(\tau;\bt) = \sum_{d \in H_2(E)} \sum_{m=0}^\infty \sum_{n=0}^\infty 
\frac{Q^d}{m!n!}
\int_{[E_{g,m+n,d}]^\vir} 
\prod_{i=1}^m \left( \sum_{k=0}^\infty \ev_i^* (t_k) \bpsi_i^k \right)
\prod_{i=m+1}^{m+n} \ev_i^* \tau
\]
Here $\tau \in H_T^\bullet(E)$, $\bt = t_0 + t_1 z + t_2 z^2 + \cdots \in H_T^\bullet(E)[z]$, $Q^d$ is the representative of $d \in H_2(E;\ZZ)$ in the Novikov ring, $\ev_i \colon E_{g,n,d} \to E$ is the evaluation map at the $i$th marked point, and the `$i$th ancestor class' $\bpsi_i \in H_T^2(E_{g,m+n,d};\QQ)$ is the pullback of the cotangent line class $\psi_i \in H^2(\Mbar_{g,m};\QQ)$ along the contraction morphism $\ct \colon E_{g,m+n,d} \to \Mbar_{g,m}$ that forgets the map and the last $n$ marked points and then stabilizes the resulting $m$-pointed curve.  We will express the total ancestor potential $\cA$, via virtual localization, in terms of the total ancestor potentials $\cA_B^{\alpha,\tw}$ of the base manifold $B$ twisted~\cite{Coates-Givental} by the $T$-equivariant inverse Euler class of the normal bundle $N^\alpha$ to the $T$-fixed locus $E^\alpha \cong B$ in $E$.  These are defined exactly as above, but replacing $E$ by $B$ and replacing the virtual class $[E_{g,m+n,d}]^\vir$ by $[B_{g,m+n,d}]^\vir \cap e_T^{-1} \big(N^{\alpha}_{g,m+n,d}\big)$ for an appropriate twisting class $N^{\alpha}_{g,m+n,d} \in K^0_T(B_{g,m+n,d})$.  

\begin{theorem}  
  \label{thm:ancestor}
    \[
    \cA(\tau;\bt) = \widehat{R(\tau)} \prod_{\alpha \in F} \cA^{\alpha,\tw}_B(u_\alpha(\tau);\bt_\alpha)
    \]
    where $\bt = \oplus_{\alpha \in F} \bt_\alpha$ with $\bt_\alpha \in H^*_T(E^\alpha)[z]$ and $\widehat{R(\tau)}$ is the quantization of the linear symplectomorphism $f \mapsto R(\tau,z) f$, and the map $\tau \mapsto \oplus_{\alpha \in F} u_{\alpha}(\tau)$ is defined by~\eqref{eq:what_is_u_alpha}.
\end{theorem}

The rest of this Section contains a proof of Theorem~\ref{thm:ancestor}.  According to~\cite[Proposition~7.3]{Giv_quantization}, the action of $\widehat{R(\tau)}$ is given by a Wick-type formula
\begin{equation}
  \label{eq:Wick_2}
  \widehat{R(\tau)} \, \left(\prod_{\alpha \in F} \cA^{\alpha, \tw}_B\big(u_\alpha(\tau);\bt_\alpha\big) \right) = \left( \exp(\Delta) \prod_{\alpha \in F} \cA_B^{\alpha, \tw} \big(u_\alpha(\tau), \bt_\alpha\big)\middle) \right|_{\bt \mapsto R(\tau,z)^{-1} \bt +  z( \Id - R(\tau,z)^{-1}) 1} 
\end{equation}
where the propagator $\Delta$, which depends on $\tau$, is defined by
\begin{align*}
  \Delta = \frac{\hbar}{2} \sum_{i,j} \sum_{\lambda,\mu} \Delta^{ij}_{\lambda,\mu} 
  \frac{\partial}{\partial t^\lambda_i}
  \frac{\partial}{\partial t^\mu_j}
  && \text{and} &&
                   \sum_{i,j} \sum_{\lambda} \Delta^{ij}_{\lambda,\mu} (-1)^{i+j} w^i z^j \phi^\lambda  = 
                   \left(\frac{R(\tau,w)^* R(\tau,z) - \Id}{w+z}\right)\phi_\mu.
\end{align*}
We will compute the $T$-equivariant total ancestor potential $\cA(\tau;\bt)$ using virtual localization, obtaining a Wick-type formula which matches precisely with \eqref{eq:Wick_2}.

We begin by factoring the fixed point loci $E^T_{g,m+n,d}$ into ``macroscopic'' pieces called stable vertices, stable edges, tails, and ends, which are defined somewhat informally as follows.  Given a $T$-fixed stable map $C \to E$ with $m+n$ marked points, forgetting the last $n$ marked points yields a stable curve $C'$ with $m$ marked points, and a stabilization morphism $C\to C'$ given by contracting the unstable components. We label points of $C$ according to their fate under the stabilization morphism: a tree of rational components of $C$ which contracts to a node, a marked point, or a regular point of $C'$ is called respectively a {\em stable edge}, a {\em tail}, or an {\em end}, while each maximal connected component of $C$ which remains intact in $C'$ is called a {\em stable vertex}. Under virtual localization, the contribution of $E^T_{g,m+n,d}$ into $\cA_E$ can be assembled from these pieces.  We will see that the contributions from stable edges, tails, and ends together give rise to the operator $R$, while the contribution of stable vertices gives $\prod_{\alpha \in F} \cA^{\alpha,\tw}_B$.

More formally, virtual localization expresses $\cA(\tau;\bt)$ as a sum over $T$-fixed strata $E_{g,n,d}^\Gamma$ in $E_{g,m+n,d}$ which are indexed by decorated graphs $\Gamma$ as in \S\ref{sec:define_Gamma}.  Consider a prestable curve $C$ with combinatorial structure $\Gamma$ and the curve $C'$ obtained from $C$ by forgetting the last $n$ marked points and contracting unstable components.  The curve $C$ can be partitioned into pieces according to the fate of points of $C$ under this process.  Those components of $C$ that survive as components of $C'$ are called \emph{stable vertices} of $C$.  Maximal connected subsets which contract to nodes of $C'$ are called \emph{stable edges} of $C$.  Maximal connected subsets which contract to a marked point of $C'$ are called \emph{tails}.  Maximal connected subsets which contract to smooth unmarked points of $C'$ are called \emph{ends}.  We denote by $\ct(\Gamma)$ the combinatorial structure of $C'$, that is, the graph $\gamma$ with vertices and edges given respectively by the stable vertices and stable edges, and with each vertex decorated by its genus and number of tails.  We arrange the sum over $\Gamma$ from virtual localization according to the stable graphs $\ct(\Gamma)$: \[ \cA(\tau;\bt) = \sum_\Gamma c_\Gamma = \sum_\gamma \sum_{\Gamma : \ct(\Gamma) = \gamma} c_{\Gamma}.  \]

We begin by analysing certain integrals over stable vertices which occur in the virtual localization formulas.  These take the form
\begin{equation}
  \label{eq:vertex_alone}
  \correlator{
    \frac{T_1(\bpsi_1)}{\chi_1-\psi_1},
    \ldots,
    \frac{T_m(\bpsi_m)}{\chi_m-\psi_m},
    \epsilon_\alpha(\psi_{m+1}),
    \ldots,
    \epsilon_\alpha(\psi_{m+n})
  }^{\alpha,\tw}_{g,m+n,d}
\end{equation}
where $T_1,\ldots,T_m$ arise from tails and/or stable edges,  $\epsilon_\alpha$ arises from ends, and each $\chi_i$ is equal to  $\frac{\chi_{\alpha \beta}}{k}$ for some $\beta \to \alpha$ and some $k \in \NN$.  Note the presence of descendant classes $\psi_i$.  Our first task is to express these vertex integrals in terms of \emph{ancestor} potentials, where no descendant classes occur.  Consider the sum
\begin{equation}
  \label{eq:vertex_sum}
  \sum_{d,n} \frac{Q^d}{n!} 
  \correlator{
    \frac{T_1(\bpsi_1)}{\chi_1-\psi_1},
    \ldots,
    \frac{T_m(\bpsi_m)}{\chi_m-\psi_m},
    \epsilon_\alpha(\psi_{m+1}),
    \ldots,
    \epsilon_\alpha(\psi_{m+n})
  }^{\alpha,\tw}_{g,m+n,d}
\end{equation}
and note the identity 
\begin{equation}
  \label{eq:psi_to_bpsi_1}
    \frac{T_i(y)}{\chi_i - x} = 
    \frac{T_i(y)}{\chi_i - y} + 
    \frac{(x - y) T_i (y)}{(\chi_i - x)(\chi_i - y)}
\end{equation}
in $H_T^\bullet(E)[\![x,y]\!]$. We can replace the insertion $\frac{T_i(\bpsi_i)}{\chi_i - \psi_i}$ in \eqref{eq:vertex_sum} first by the right-hand side of \eqref{eq:psi_to_bpsi_1} with $x=\psi_i$ and $y=\bpsi_i$, and then by
\begin{equation}
  \label{eq:psi_to_bpsi_2}
  \frac{T_i(\bpsi_i)}{\chi_i - \bpsi_i} + 
  \sum_{n,d} \frac{Q^d}{n!}  \sum_{\mu,\nu} \phi_\mu
  \correlator{
    \frac{\phi^\mu}{\chi_i-\psi_1},
    \epsilon_\alpha(\psi_2),\ldots,\epsilon_\alpha(\psi_{n+1}),\phi_\nu}^{\alpha,\tw}_{0,n+2,d}
  \left( \phi^\nu,\frac{T_i(\bpsi_i)}{\chi_i-\bpsi_i} \right)^{\alpha,\tw} 
\end{equation}
Here we used the fact, first exploited by Getzler~\cite{Getzler:jet} and Kontsevich--Manin~\cite{Kontsevich--Manin}, that $\psi_i - \bpsi_i$ is Poincar\'e-dual to the virtual divisor which is total range of the gluing map
\[
\bigsqcup_{\substack{n_1+n_2=n\\d_1+d_2=d}}
E_{0,n_1+2,d_1} \times_E E_{g,m-1+n_2+1,d_2}
\longrightarrow E_{g,m+n,d}
\]
that attaches a genus-zero stable map carrying the $i$th marked point and $n_1$ marked points with insertions in $\{m+1,\ldots,m+n\}$ to a genus-$g$ stable map carrying the marked points $1,\ldots,m$ omitting $i$, and $n_2$ marked points with insertions in $\{m+1,\ldots,m+n\}$.  The insertion \eqref{eq:psi_to_bpsi_2} simplifies to
\begin{equation}
  \label{eq:psi_to_bpsi_3}  
  \frac{S^{\alpha,\tw}(u_\alpha(\tau),\chi_i) T(\bpsi_i)}{\chi_i-\bpsi_i} 
\end{equation}
using~\cite[equation~2]{Giv_Frob} and the fact that the contribution $\epsilon_\alpha$ from ends here coincides with the contribution $\epsilon_\alpha$ from ends in genus zero which occurred in equation~\eqref{eq:what_is_u_alpha}.  In what follows we will see that whenever contributions from tails and stable edges are expressed in terms of $\bpsi_i$-classes, the transformations $S^{\alpha,\tw}(u_\alpha(\tau),\chi_i)$ will occur.  We call these transformations the \emph{dressing factors}.

Consider the following generating function which we call the \emph{mixed potential}:
\[  
\cA(\bepsilon;\bt) = \exp \left(\sum_{g=0}^\infty \sum_{m,n=0}^\infty \sum_d \frac{\hbar^{g-1} Q^d}{m! n!}
  \correlator{\bt(\bpsi_1),\ldots,\bt(\bpsi_m);\bepsilon(\psi_{m+1}),\ldots,\bepsilon(\psi_{m+n})}_{g,m+n,d}\right)
\]
Note that specializing $\bepsilon \in \cH_+$ to $\tau \in H^\bullet_T(E) \subset \cH_+$ recovers the ancestor potential $\cA(\tau;\bt)$.  The notion of mixed potential makes sense for a general Gromov--Witten-type theory, including the Gromov--Witten theory of arbitrary target spaces (equivariant or not), twisted Gromov--Witten theory, etc., and Proposition~\ref{pro:mixed_from_ancestor} below expresses the mixed potential of such a theory in terms of the ancestor potential for that theory.  The argument just given showed that the sum \eqref{eq:vertex_sum} can be expressed, by including appropriate dressing factors, in terms of the mixed potential associated to the twisted Gromov--Witten theory of the $T$-fixed locus $E^\alpha$.  

\begin{proposition}
  \label{pro:mixed_from_ancestor}
  \[
  \cA(\bepsilon;\bt) = 
  \cA\Big(u(\bepsilon);\bt+\big[S(u(\bepsilon),z)(\bepsilon(z)-u(\bepsilon))\big]_+\Big)
  \]
  where $u(\bepsilon)\in H$ is characterized by
  \[
  \big[S(u(\bepsilon),z) \big(\bepsilon(z) - u(\bepsilon)\big)\big]_+ \in z\cH_+.
  \]
\end{proposition}
  
\begin{proof}
  Set $\by(z) = \bepsilon(z) - u(\bepsilon)$, so that:
  \[
  \cA(\bepsilon;\bt) = \exp \left(\sum_{g=0}^\infty \sum_{k,l,m=0}^\infty \sum_d  \frac{\hbar^{g-1} Q^d}{k! l! m!} 
    \correlator{\bt(\bpsi_1),\ldots,\bt(\bpsi_m);\by(\psi_{m+1}),\ldots,\by(\psi_{m+k}),u(\bepsilon),\ldots,u(\bepsilon)}_{g,k+l+m,d}\right)
  \]
  Then consider the morphism $X_{g,m+k+l,d} \to \Mbar_{g,m+k}$ forgetting the map and the last $l$ marked points and then stabilizing; the Getzler/Kontsevich--Manin ancestor-to-descendant argument discussed above then gives:
  \[
  \cA(\bepsilon;\bt) = \exp \left(\sum_{g=0}^\infty \sum_{k,l,m=0}^\infty \sum_d  \frac{\hbar^{g-1} Q^d}{k! l! m!}
    \correlator{\bt(\bbpsi_1),\ldots,\bt(\bbpsi_m),\bx(\bpsi_{m+1}),\ldots,\bx(\bpsi_{m+k}); u(\bepsilon),\ldots,u(\bepsilon)}_{g,k+l+m,d}\right)
  \]
  where $\bx(z) = \big[S({u(\bepsilon)},z) \by(z)\big]_+$ and the classes $\bbpsi_i$ differ from the ancestor classes $\bpsi_i$ by being lifts of $\psi$-classes from $\Mbar_{g,m}$ rather than from $\Mbar_{g,m+k}$.  For $\psi$-classes on Deligne--Mumford spaces, we have:
  \begin{align*}
    \left(\prod_{r=m+1}^{m+k} \psi_r\right) \cup (\psi_i - \pi^* \psi_i) = 0 && \text{for $i = 1,2,\ldots,m$}
  \end{align*}
  where $\pi \colon \Mbar_{g,m+k} \to \Mbar_{g,m}$ is the map that forgets the last~$k$ marked points and then stabilizes.  This is because $\psi_i - \pi^* \psi_i$ is the divisor in $\Mbar_{g,m+k}$ given by the image of the gluing map
  \[
  \bigsqcup_{k_1+k_2=k}
  \Mbar_{g,m-1+k_2+1} \times \Mbar_{0,k_1+1} \to \Mbar_{g,m+k}
  \]
  where the second factor carries $k_1$ marked points with indices in $\{m+1,\ldots,m+k\}$, and on this divisor the product $\prod_{r=m+1}^{m+k} \psi_r$ vanishes for dimensional reasons.  Thus:
  \begin{align*}
    \left(\prod_{r=m+1}^{m+k} \bpsi_r\right) \cup (\bpsi_i - \bbpsi_i) = 0 && \text{for $i = 1,2,\ldots,m$}
  \end{align*}
  and, since $\bx(z)$ is divisible by $z$ -- this is exactly how we chose $u(\bepsilon)$ -- we may replace $\bbpsi_i$s by $\bpsi_i$s in our expression for $\cA(\bepsilon,\bt)$ above.  Polylinearity then gives the Proposition.
\end{proof}

Now we apply Proposition~\ref{pro:mixed_from_ancestor} to the twisted Gromov--Witten theory of the $T$-fixed locus $E^\alpha$, so that $S$ there is $S^{\alpha,\tw}$.  We will show that if we set $\bepsilon$ equal to the end contribution $\epsilon_\alpha$ in the vertex integrals \eqref{eq:vertex_sum}, then $u(\bepsilon)$ becomes the block-canonical coordinate $u_\alpha(\tau)$ defined in Proposition~\ref{pro:R}, and that
\begin{equation}
  \label{eq:end_contribution_matches_Wick}
  \big[S^{\alpha,\tw}(u_\alpha(\tau),z)(\epsilon_\alpha(z)-u_\alpha(\tau))\big]_+ = z \big(\Id - R(\tau,z)^{-1}\big) 1.
\end{equation}
Note that this precisely matches the contribution to Wick's Formula in equation~\eqref{eq:Wick_2}.  Since $S^{\alpha,\tw}(u_\alpha(\tau),z)$ is a power series in $z^{-1}$, we have:
\[
\Big[ S^{\alpha,\tw}(u_\alpha(\tau),z) \big[ A \big]_+ \Big]_+ = \Big[ S^{\alpha,\tw}(u_\alpha(\tau),z) \, A \Big]_+
\]
for any $A$.  From \eqref{eq:end_from_J}, we have
\[
\epsilon_\alpha(z) = z - \Big[ z S(\tau,{-z})^* 1\big|_{E^\alpha} \Big]_+
\]
and Proposition~\ref{pro:R} gives
\[
S(\tau,-z)^* 1\big|_{E^\alpha} = S^{\alpha, \tw}(u_\alpha(\tau),{-z})^* \Big( R(\tau,{-z})^* 1  \big|_{E^\alpha}\Big).
\]
We compute:
\[
\big[S^{\alpha,\tw}(u_\alpha(\tau),z)(\epsilon_\alpha(z)-u_\alpha(\tau))\big]_+ = z + u_\alpha(\tau)  - z R(\tau,{-z})^* 1  \big|_{E^\alpha} - u_\alpha(\tau)
\]
where the first two terms are $\big[S^{\alpha,\tw}(u_\alpha(\tau),z) z\big]_+$, the last term is $\big[S^{\alpha,\tw}(u_\alpha(\tau),z) u_\alpha(\tau)\big]_+$, and we used $S^{\alpha,\tw}(u_\alpha(\tau),z) S^{\alpha,\tw}(u_\alpha(\tau),{-z})^* = \Id$.  The right-hand side here lies in $z \cH_+$, and therefore $u(\epsilon_\alpha) = u_\alpha(\tau)$, as claimed; furthermore, since $R(\tau,-z)^* R(\tau,z) = \Id$ we obtain \eqref{eq:end_contribution_matches_Wick}.

We return now to the integrals \eqref{eq:vertex_alone} over stable vertices.  Here \emph{tails} -- which carry one non-forgotten marked point each -- give rise to insertions of
\[
\bt(\bpsi) + \sum_{\alpha \in F} \sum_{\beta: \beta \to \alpha} \sum_{k=1}^\infty  
\frac{1}{\frac{\chi_{\alpha \beta}}{k} - \psi}
T^k_{\beta,\alpha} \bt(\bpsi)
\] 
where the linear maps $T^k_{\beta,\alpha}$ were defined in the proof of Proposition~\ref{pro:R} and $\bt$ is the argument of the ancestor potential~\eqref{eq:ancestor_potential}.  Applying equations (\ref{eq:psi_to_bpsi_1}--\ref{eq:psi_to_bpsi_3}), we can write this in terms of $\bpsi$ only by including dressing factors:
\[
\bt(\bpsi) + \sum_{\alpha \in F} \sum_{\beta: \beta \to \alpha} \sum_{k=1}^\infty  
\frac{1}{\frac{\chi_{\alpha \beta}}{k} - \bpsi}
S^{\alpha,\tw}\big(u_\alpha(\tau),{\textstyle \frac{\chi_{\alpha \beta}}{k}}\big) 
T^k_{\beta,\alpha} \bt(\bpsi)
\] 
From \eqref{eq:R_explicit} and $R(\tau,-z)^* R(\tau,z) = \Id$, we see that this is $\big(R(\tau,z)^{-1} \bt(z) \Big)\big|_{E^\alpha}$ with $z = \bpsi$.  Again, this precisely matches the contribution to Wick's Formula in equation~\eqref{eq:Wick_2}. 

In the localization formula for the ancestor potential, the edge contributions occur in the form
\[
\sum_{\beta \to \alpha}
\sum_{\beta' \to \alpha'}
\sum_{k=1}^\infty
\sum_{k'=1}^\infty
\frac{E^{k,k'}_{\beta\to\alpha,\beta'\to\alpha'}}
{
\Big(\frac{\chi_{\alpha\beta}}{k} - \psi\Big)
\Big(\frac{\chi_{\alpha'\beta'}}{k'} - \psi'\Big)
},
\]
each functioning as two vertex insertions (which may be to the same vertex or different vertices).  Here $\psi$ and $\psi'$ are the $\psi$-classes on the vertex moduli spaces at the corresponding marked points, and the class $E^{k,k'}_{\beta\to\alpha,\beta'\to\alpha'} \in H^\bullet_T(B) \otimes H^\bullet_T(B)$ is pulled back to these moduli spaces by the product $\ev \times \ev'$ of the corresponding evaluation maps.  This expression participates, in the same role, in localization formulas for the genus-zero quantity $V(\tau,w,z)$.
As before, we can replace $\psi$-classes by $\bpsi$-classes provided that we introduce dressing factors:
\[
\sum_{\beta \to \alpha}
\sum_{\beta' \to \alpha'}
\sum_{k=1}^\infty
\sum_{k'=1}^\infty
\frac{
  S^{\alpha,\tw}\big(u_\alpha(\tau),{\textstyle \frac{\chi_{\alpha \beta}}{k}}\big) \otimes 
  S^{\alpha',\tw}\big(u_{\alpha'}(\tau),{\textstyle \frac{\chi_{\alpha' \beta'}}{k'}}\big) 
  \Big(E^{k,k'}_{\beta\to\alpha,\beta'\to\alpha'}\Big)}
{
\Big(\frac{\chi_{\alpha\beta}}{k} - \bpsi\Big)
\Big(\frac{\chi_{\alpha'\beta'}}{k'} - \bpsi'\Big)
}
=: E_{\alpha,\alpha'}(\bpsi,\bpsi').
\]
Computing $V(\tau,w,z)$ by virtual localization, as in the proof of Proposition~\ref{pro:R}, and applying the identities 
\[
V^{\alpha,\tw}\Big(u_\alpha,z,{\textstyle\frac{\chi_{\alpha\beta}}{k}}\Big) = \frac{S^{\alpha,\tw}\Big(u_\alpha,z\Big)^* S^{\alpha,\tw}\Big(u_\alpha,\frac{\chi_{\alpha\beta}}{k}\Big)}{z+\frac{\chi_{\alpha\beta}}{k}}
\]
for each $\alpha \in F$, $\beta \to \alpha$, and $k \in \{1,2,\ldots\}$ yields
\[
V(\tau,z,z') = S_\block(\tau,z)^* \circ \left( \frac{\Id}{z+z'} + \bigoplus_{\alpha,\alpha' \in F} E_{\alpha,\alpha'}(z,z')
\right) \circ S_\block(\tau,z')
\]
where we regard the bivector $E_{\alpha,\alpha'}$ as an operator via the twisted Poincar\'e pairing.  But
\[
V(\tau,z,z') = \frac{S(\tau,z)^* S(\tau,z')}{z+z'} = \frac{S_\block(\tau,z)^* R(\tau,z)^* R(\tau,z') S_\block(\tau,z')}{z+z'}
\]
from \eqref{eq:JSV} and Proposition~\ref{pro:R}, and we conclude that
\[
\bigoplus_{\alpha,\alpha' \in F} E_{\alpha,\alpha'}(z,z') = \frac{R(\tau,z)^* R(\tau,z') - \Id}{z+z'}.
\]
In other words, the edge contributions $E_{\alpha,\alpha'}$ (which include dressing factors) are precisely what is inserted by the propagator in Wick's formula \eqref{eq:Wick_2}.  This completes the proof of Theorem~\ref{thm:ancestor}.

\begin{remark}
  Consider a K\"ahler manifold $E$ equipped with the action of a torus $T$, with no further assumptions about the structure of the fixed point manifold $E^T$ or the one-dimensional orbits. Theorem~\ref{thm:ancestor} can be extended to this general situation. First, virtual localization in genus zero shows, as in \S\ref{sec:genus_zero}, that the fundamental solution matrix $S_E(\tau,z)$ can be factored as
  \[ 
  S_E(\tau,z)=R(\tau,z) \, S_{E^T}^{\tw}\big(u(\tau),z\big)
  \]
  where $R$ is a power series in $z$, and $\tau \mapsto u(\tau)$ is a certain non-linear diffeomorphism between the parameter spaces $H^\bullet_T(E)$ and $H^\bullet_T(E^T)$ which is defined over the field of fractions of the coefficient ring $H^\bullet_T(\{ \text{\rm point} \} )$ tensored with suitable Novikov ring (and can be specified in terms of genus-zero Gromov--Witten invariants).
 Then, virtual localization in all genera shows that
 \[  
 \cA_E(\tau; {\mathbf t}) = \widehat{R (\tau)} \, \cA_{E^T}^{\tw}\big(u(\tau); {\mathbf t}|_{E^T}\big).
 \]
 However, we refrained from phrasing our proof of Theorem~\ref{thm:ancestor} in this abstract setting.  In the context of general torus actions, one-dimensional $T_{\mathbb C}$-orbits -- called legs in \S\ref{sec:define_Gamma} --  may depend on parameters. Some foundational work is needed here -- a systematic description of leg moduli spaces and their virtual fundamental cycles -- and establishing these details, which are unimportant to the essence of our argument, would carry us too far away from the current aim. This is the only reason why we limit the proof of Theorem~\ref{thm:ancestor} to the case of toric bundles.     
 \end{remark}

\section{The Non-Equivariant Limit}
\label{sec:non-equivariant_limit}

As discussed in the Introduction, combining Theorem~\ref{thm:ancestor} with the Quantum Riemann-Roch theorem yields
\begin{equation}
  \label{eq:equivariant_formula}
  e^{F^{\eq}(\tau)}\A^{\eq}(\tau) = e^{F^{\eq}(\tau)-\sum_{\alpha} F^{\alpha,\tw}_{B}(u_\alpha(\tau))} \widehat{R(\tau)} \widehat{S_\block(\tau)} \widehat{\Gamma_\block^{-1}}
  \prod_{\alpha \in F} \cD_B
\end{equation}
where $\Gamma_\block = \oplus_{\alpha \in F} \Gamma_{\alpha}$ and $S_{\block}(\tau) = \oplus_{\alpha\in F} S^{\alpha,\tw}\big(u_{\alpha}(\tau)\big)$.  Thus the $T$-equivariant ancestor potential $\cA^\eq(\tau)$ for $E$ is obtained from $\cD_B^{\otimes |F|}$ by the application of quantized loop group operators.  We need to show that the same is true for the \emph{non-equivariant} ancestor potential $\cA_E(\tau)$ of $E$, and so we need to analyse the non-equivariant limit of \eqref{eq:equivariant_formula}.

Note first that the products $S^{\alpha,\tw}(u_{\alpha}) \, \Gamma_{\alpha}^{-1}$ can be Birkhoff-factorized as:
\[ 
S^{\alpha, \tw}(u_{\alpha}, z) \, \Gamma_{\alpha}^{-1}(z) = R_{\alpha}(u_{\alpha}, z) \,  S_B\big(\tau^*_{\alpha}(u_{\alpha}), z\big)
\]
where $R_{\alpha}$ is an operator-valued power series in $z$, and $u\mapsto \tau_\alpha^*(u)$ is a non-linear change of coordinates\footnote{This is the ``mirror map''.} on $H^{\bullet}(B)$. Indeed, the Quantum Riemann--Roch theorem~\cite{Coates-Givental} implies that $\Gamma_{\alpha}^{-1}$ transforms the overruled Lagrangian cone $\cL_B$ defined by the genus-zero Gromov--Witten theory of $B$ into the overruled cone $\cL^{\alpha,\tw}$ for the twisted theory: $\cL^{\alpha,\tw} = \Gamma^{-1}_{\alpha}\cL_B$. Thus, the operator $R_{\alpha}$ on the space $\cH^B_{+}=H_B[[z]]$ is obtained from the following composition: 
\[
\xymatrix{
  R_{\alpha}^{-1} \colon  \cH^B_{+} \ar[rr]^{S^{\alpha,\tw}(u)^{-1}} && 
  T_u\cL^{\alpha,\tw} \ar[rr]^{\Gamma_{\alpha}} &&
  T_{\tau^*_{\alpha}}\cL_B \ar[rr]^{S_B(\tau^*_{\alpha})} &&
  \cH^B_{+}}
\]
and so
\[ 
R(\tau) \, S_{\block}(\tau) \, \Gamma_{\block}^{-1} = R'(\tau) \left( \bigoplus_{\alpha \in F}  S_B\big(\tau^*_{\alpha}(u_\alpha(\tau))\big)\right)
\] 
where $R'(\tau) := R(\tau) \Big(\!\oplus_{\alpha\in F}R_{\alpha}\big(u_{\alpha}(\tau)\big)\Big)$ is an operator-valued power series in $z$. Consequently, for some scalar functions $c$ and
$\tilde{c}$, we have:
\begin{align}
  \label{eq:refactored}
  \cA^{\eq}_E(\tau) = e^{c(\tau)} \widehat{R'(\tau)} \prod_{\alpha\in F} \cA_B(\tau^*_{\alpha}) && \text{and} &&
  \cD_E^{\eq}= e^{\tilde{c}(\tau)} \widehat{S_E(\tau)^{-1}} \widehat{R'(\tau)} \prod_{\alpha\in F} \widehat{S_B(\tau^*_{\alpha})} \cD_B.
\end{align}
Note that the relation between $\tau \in H_T^{\bullet}(E)$ and $\oplus_{\alpha \in F} \tau^*_{\alpha} \in H_T^{\bullet}(E^T)$ here is a complicated change of variables, given by composing the Dijkgraaf--Witten maps $\tau \mapsto u_{\alpha} (\tau)$ with the mirror maps $u_{\alpha}\mapsto \tau^*_{\alpha}(u_{\alpha})$.
      
In this section we show that, at least for some range\footnote{It suffices to establish that the non-equivariant limit exists for a single value of $\tau$, since $\cD_E^{\eq}$ and $\cD_B$ are independent of $\tau$.} of $\tau$, the operator $R'$ has a well-defined non-equivariant limit. The ingredients here are Brown's mirror theorem for toric bundles~\cite{Brown_toric}, which provides a certain family of elements $I_E(t,\tau_B,{-z})$ on the Lagrangian cone $\cL_E$ for the $T$-equivariant Gromov--Witten theory of $E$, and an analysis of the stationary phase asymptotics of the oscillating integrals that form the mirror for the toric fiber $X$ of $E$.  Since all other ingredients in \eqref{eq:refactored} -- $\cA^{\eq}_E$, $\cA_B$, $\cD^{\eq}_E$, $\cD_B$, $S_E$, and $S_B$ -- also have well-defined non-equivariant limits, it follows that the same is true for $\tilde{c}$ and $c$.
Thus the entire formula \eqref{eq:refactored} relating $\cD_E$  to $\cD_B^{\otimes |F|}$ by the action of a quantized loop group operator admits a non-equivariant limit.   Finally, in \S\ref{sec:grading_preserving} we complete the proof of Theorem~\ref{thm:main}, by checking that, in the non-equivariant limit, this loop group operator is grading-preserving.

\subsection{The $I$-Function of $E$}
\label{sec:yet_more_toric_bundles}

Recall that our toric bundle $E \to B$ is obtained from the total space of a direct sum of line bundles $L_1 \oplus \cdots \oplus L_N \to B$ by fiberwise symplectic reduction for the action of a subtorus $K = (S^1)^k$ of $T = (S^1)^N$.  Let $\mu$ denote the moment map for the action of $K$ on the total space of $L_1 \oplus \cdots \oplus L_N \to B$, so that $E = \mu^{-1}(\omega)/K$. The quotient map $\mu^{-1}(\omega) \to E$ exhibits $\mu^{-1}(\omega)$ as a principal $K$-bundle over $E$.  Since $K = (S^1)^k$, this defines $k$~tautological $S^1$-bundles over $E$, each of which carries an action of $T$.  Let $P_1,\ldots,P_k \in H^2_T(E;\ZZ)$ be the $T$-equivariant first Chern classes of the corresponding anti-tautological bundles, and let $p_1,\ldots,p_k$ be the restrictions of $P_1,\ldots,P_k$ to the fiber $X$.  Without loss of generality, by changing the identification of $K$ with $(S^1)^k$ if necessary, we may assume that the classes $p_1,\ldots,p_k$ are ample.  The classes $p_1,\ldots,p_k$ generate the $T$-equivariant cohomology algebra $H^\bullet_T(X)$; let
\begin{equation}
  \label{eq:monomials}
  \Delta_\beta(p_1,\ldots,p_k)
\end{equation}
be monomials in $p_1,\ldots,p_k$, indexed by $\beta$, that together form a basis for $H^\bullet_T(X)$.

Let $\lambda_1,\ldots,\lambda_N$ denote the first Chern classes of the $N$~anti-tautological bundles on $BT = (\CC P^\infty)^N$, so that $R_T := H^\bullet_T(\{\text{\rm point}\};\QQ) = \QQ[\lambda_1,\ldots,\lambda_N]$.  The $T$-equivariant cohomology algebra $H^\bullet_T(E)$ has basis $\Delta_\beta(P_1,\ldots,P_k)$ over $H^\bullet(B) \otimes R_T$ (cf.~\cite{US}).  Let $\Lambda_j$ denote the first Chern class of the dual bundle $L_j^\vee$, so that the $T$-equivariant first Chern class of $L_j$ is $-\Lambda_j - \lambda_j$, and define:
\begin{align*}
  u_j = \sum_{i=1}^{k} m_{ij} p_i - \lambda_j 
  && U_j=\sum_{i=1}^{k} m_{ij} P_i -\Lambda_j - \lambda_j && 1 \leq j \leq N.
\end{align*}
Then $u_j$ is the $T$-equivariant cohomology class Poincar\'e dual to the $j$th toric divisor in $X$, and $U_j$ is the $T$-equivariant cohomology class Poincar\'e dual to the $j$th toric divisor in $E$.  

Let $J^B(\tau_B,z)$ denote the $J$-function of $B$, and write: 
\[
J^B(\tau_B,z) =\sum_{\beta\in \Eff(B)} J^B_\beta(\tau_B,z) Q_B^\beta.
\]
Here $\tau_B\in H^\bullet(B)$, which can be written as $\tau_B=\sum_b \tau_b\phi_b$ after choosing a basis $\phi_1,\ldots,\phi_{\rk H^\bullet(B)}$ for $H^\bullet(B)$.

The $I$-function of $E$ is:
\begin{equation}
  \label{eq:what_is_I}
  I_E(t,\tau_B,z):=e^{Pt/z} \sum_{d\in \ZZ^k} \sum_{\beta\in \Eff(B)} J^B_\beta(\tau_B,z) Q_B^\beta q^d e^{dt}
  \prod_{j=1}^N 
  \frac{\prod_{m=-\infty}^{0} (U_j+mz)}{\prod_{m=-\infty}^{U_j(d, \beta)} (U_j+mz)}
\end{equation}
where:
\begin{align*}
  &t:=(t_1,\dots, t_k) &
  &d  = (d_1,\ldots,d_k) &
  &Pt:=\sum_{i=1}^k P_it_i \\ 
  &dt:=\sum_{i=1}^k d_it_i &
  &q^d := q_1^{d_1} \cdots q_k^{d_k} &
  &U_j(d, \beta) :=\sum_{i=1}^k d_im_{ij}-\int_{\beta}\Lambda_j
\end{align*}
We have that:
\begin{equation}
  \label{eq:I_mod_Q}
  I_E(t,\tau_B,z) = z e^{Pt/z} e^{\tau_B/z} + \bigO(Q)
\end{equation}
Brown proves~\cite[Theorem~1]{Brown_toric} that $I_E(t,\tau_B,{-z})$ lies on the Lagrangian cone $\cL_E$ defined by the $T$-equivariant Gromov--Witten theory of $E$.  

Recall the monomials $\Delta_\beta$ defined in \eqref{eq:monomials}.  The classes 
\begin{equation}\label{def:basis_H(E)}
\Delta_\beta(P_1,\ldots,P_k) \phi_b,
\end{equation}
where $\{\phi_b\}$ is a basis for $H^\bullet(B)$, form a basis for $H^\bullet_T(E)$, and this basis has a well-defined non-equivariant limit.  Elements of this basis are indexed by pairs $(\beta, b)$. Denote by $T$ the following matrix whose column vectors are also indexed by pairs $(\beta,b)$:
\begin{equation}
  \label{eq:what_is_T}
  T(z) = 
  \begin{bmatrix}
    \Delta_\beta\Big(z \frac{\partial}{\partial {t_1}},\ldots,z \frac{\partial}{\partial {t_k}}\Big) \partial_{\phi_b} I_E(t,\tau_B,z)
  \end{bmatrix}.
\end{equation}
More precisely, the column of $T(z)$ corresponding to the pair $(\beta,b)$ is $\Delta_\beta\Big(z \frac{\partial}{\partial {t_1}},\ldots,z \frac{\partial}{\partial {t_k}}\Big) \partial_{\phi_b} I_E(t,\tau_B,z)$.

The columns of $T(-z)$ form a basis for $T_{I(t,\tau_B,{-z})} \cL_E$ over the ring $\Nov\{z\}$ of power series in $q_1,\ldots,q_k$ and Novikov variables of $B$ with coefficients which are polynomials of $z$: see \eqref{eq:I_mod_Q}.  For an appropriate value of $\tau$, determined by $t$ and $\tau_B$, the columns of $S_E(\tau,{-z})^*$ form another such basis, which consists of series in $z^{-1}$.  Expressing the columns 
of $T$ in terms of columns of $S$ yields
\[
T({-z}) = S_E(\tau,{-z})^* L(\tau, z);
\]
where $L$ is a matrix with entries in $\Nov\{z\}$; this is the Birkhoff factorization of $T$.  Since all other ingredients in this identity admit a non-equivariant limit, $L(\tau, z)$ does too.  

The products 
\[
\frac{\prod_{m=-\infty}^{0} (U_j+mz)}{\prod_{m=-\infty}^{U_j(d, \beta)} (U_j+mz)}
\]
that occur in \eqref{eq:what_is_I}, and thus in \eqref{eq:what_is_T}, are rational functions of $z$.  When $\lambda_j \ne 0$ for each $j$, we expand these as Laurent series near $z=0$.  We saw in \S\ref{sec:genus_zero} that the same operation applied to $S_E(\tau,z)$ yields
\[
S_E(\tau,z) =  R(\tau,z) S_\block(\tau,z).
\]
Thus
\[
T(-z) \sim  S_\block(\tau,z)^{-1} R(\tau,z)^{-1} L(\tau, z)
\]
where $\sim$ denotes the expansion near $z=0$. 
In the next two subsections we will show that the expansion of $\Gamma_\block (z) \, T({-z})$ near $z=0$ has a non-equivariant limit, by identifying this expansion with the stationary phase expansion of certain oscillating integrals.

\subsection{Oscillating Integrals}

Consider $W = \sum_{j=1}^{N} (x_j + \lambda_j \log x_j)$, and oscillating integrals 
\[
\int_{\gamma} e^{W/z} \,
\frac{
\prod_{j=1}^{N} d \log x_j
}{
\prod_{i=1}^{k} d \log (q_ie^{t_i})  
}
\]
where $\gamma$ is a cycle in the subvariety of $(\Cstar)^N$ defined by
\begin{align}
  \label{eq:fiber}
  \prod_{j=1}^{N} x_j^{m_{ij}} = q_i e^{t_i} && 1 \leq i \leq k
\end{align}
given by downward gradient flow of $\Re(W/z)$ from a critical point of $W$. Such oscillating integrals, over an appropriate set of cycles $\gamma$, together form the mirror to the $T$-equivariant quantum cohomology of the toric manifold $X$~\cite{Giv_toric}.  We now relate these integrals to the $q$-series $I_E(t,\tau_B,z)$ by expanding the integrand as a $q$-series, following~\cite{Brown_toric}.

Given a $T$-fixed point $\alpha$ on $X$, one can solve the equations \eqref{eq:fiber} for $x_i$,~$i \in \alpha$, in terms of $x_j$,~$j \not \in \alpha$:
\begin{align}
  \label{eq:solve_for_x_i}
  x_i = \prod_{l=1}^{k} (q_l e^{t_l})^{(m_\alpha^{-1})_{il}}  \prod_{j \not \in \alpha} x_j^{-(m_\alpha^{-1} m)_{ij}}
  &&
  i \in \alpha
\end{align}
where $m_\alpha$ is the $k \times k$ submatrix of $(m_{ij})$ given by taking the columns in $\alpha$; this defines a chart on the toric mirror \eqref{eq:fiber}.  In this chart the integrand $e^{W/z}$ becomes 
\[
\Phi_\alpha := \left(\prod_{i=1}^k (q_i e^{t_i})^{\alpha^*(p_i)/z}\right)
e^{\sum_{j \not \in \alpha} (x_j - \alpha^*(u_j) \log x_j )/z}
\sum_{\substack{
    d \in \ZZ^k : \\
    \text{$u_j(d) \geq 0$ for $j \in \alpha$}
  }}
  (q_1 e^{t_1}) ^{d_1} \cdots (q_k e^{t_k})^{d_k} 
  \frac{\prod_{j \not \in \alpha} x_j^{-u_j(d)}}
  {\prod_{j \in \alpha} u_j(d)!z^{u_j(d)} }
\]
where $u_j(d) = \sum_{i=1}^k d_i m_{ij}$; cf.~the proof of \cite[Theorem 3]{Brown_toric}.  Note that $u_j(d)$ is the value of the cohomology class $u_j \in H^2(X)$ on the element $d \in H_2(X)$ such that $p_i(d) = d_i$, and that our ampleness assumption guarantees that all $d_i$ are non-negative.  We thus consider
\[
\cI_\alpha(q e^t,z,\lambda) = \int_{(\RR_+)_{j \not \in \alpha}} \Phi_\alpha \bigwedge_{j \not \in \alpha} d \log x_j
\]
as a $q$-series of oscillating integrals with phase function ${\sum_{j \not \in \alpha} (x_j - \alpha^*(u_j) \log x_j )}$ and monomial amplitudes.  Replacing the variables $\lambda_j$ by the differential operator $\lambda_j + z \partial_{\Lambda_j}$, we can consider $\cI_\alpha$ as an operator to be applied to the $J$-function of the base $B$.  According to the computation in~\cite[Section~5]{Brown_toric}, this gives
\begin{equation}
  \label{eq:integration_by_parts}
  \left(\prod_{i=1}^k q_i^{-\alpha^*(P_i)/z}\right) \cI_\alpha(q e^t,z,\lambda + z \partial_\Lambda) J_B(\tau_B,z) = 
  \alpha^* I_E (t,\tau_B,z) \prod_{j \not \in \alpha} \int_0^\infty e^{(x - \alpha^*(U_j) \log x)/z} d \log x.
\end{equation}
Applying the differential operators in \eqref{eq:what_is_T}, we get expressions in terms of oscillating integrals for each entry $\Delta_\beta\Big(z \frac{\partial}{\partial {t_1}},\ldots,z \frac{\partial}{\partial {t_k}}\Big) \partial_{\phi_b} I_E$ of $T(z)$.  The stationary phase asymptotics of the oscillating integrals on the right (at the unique critical points $x = \alpha^*(U_j)$ of the phase functions) combine to give $(2\pi z)^{\dim_{\CC} X/2}\Gamma_\alpha(z)$; see~\cite[Section~5.2]{Brown_toric}.   Consequently
\begin{equation}
  \label{eq:entries_of_T}
  \Big(\Gamma_\block({-z}) \, T(z) \Big)^{\beta,b}_\alpha \sim
  \left(\prod_{i=1}^k q_i^{-\alpha^*(P_i)/z}\right) 
  \frac{\Delta_\beta\Big({z \textstyle \frac{\partial}{\partial {t_1}},\ldots,z \frac{\partial}{\partial {t_k}}}\Big)}{(2 \pi z)^{\frac{1}{2} \dim_{\CC} X}}
  \cI_\alpha(q e^t,z,\lambda + z \partial_\Lambda)\partial_{\phi_b}J_B(\tau_B,z),
\end{equation}
that is, the expansions near $z=0$ of the entries of $\Gamma_\block(-z) \, T(z)$ coincide with stationary phase asymptotics of the oscillating integrals on the right-hand side of \eqref{eq:entries_of_T}. We now show that these stationary phase asymptotics admit a non-equivariant limit.

\subsection{Stationary Phase Asymptotics}

The expression
\begin{equation}
  \label{eq:with_amplitudes}
  \Delta_\beta\Big({z \textstyle \frac{\partial}{\partial {t_1}},\ldots,z \frac{\partial}{\partial {t_k}}}\Big)
  \int_{\gamma} e^{W/z} \,
  \frac{
    \prod_{j=1}^{N} d \log x_j
  }{
    \prod_{i=1}^{k} d \log (q_ie^{t_i}) 
  }
\end{equation}
is an oscillating integral with the phase function $\sum_{j=1}^N (x_j + \lambda_j \log x_j)$ over a Lefschetz thimble $\gamma$ in the complex torus with equations
\begin{align}
  \label{eq:fiber_again}
  \prod_{j=1}^{N} x_j^{m_{ij}} = q_i e^{t_i} && 1 \leq i \leq k.
\end{align}
In the classical limit $q \to 0$, this torus degenerates into a union of coordinate subspaces  (\foreignlanguage{russian}{противотанковый еж})  $\bigcup_{\alpha \in F}\CC^{n-k}_\alpha$, where $\CC^{n-k}_\alpha$ is the subspace of $\CC^n$ given by the equations $x_j = 0$,~$j \in \alpha$.  The equations for critical points of $W$ under the constraints \eqref{eq:fiber_again} in the coordinate chart $\{x_j : j \not \in \alpha\}$ take the form:
\begin{align*}
  0 = x_j - \alpha^*(u_j) + \text{terms involving positive powers of $q$} 
  &&
  j \not \in \alpha.
\end{align*}
In the classical limit $q \to 0$, exactly one of these critical points approaches the critical point
\begin{align}
  \label{eq:unperturbed_critical_point}
  x_j = \alpha^*(u_j) && j \not \in \alpha
\end{align}
of the phase function $\sum_{j \not \in \alpha} (x_j - \alpha^*(u_j) \log x_j)$ on $\CC^{n-k}_\alpha$.  Call this the \emph{$\alpha$th critical point} of $W$.  On the right-hand side of \eqref{eq:entries_of_T}, we first expanded the oscillating integral \eqref{eq:with_amplitudes} as a $q$-series and then took termwise stationary phase asymptotics at the critical point \eqref{eq:unperturbed_critical_point}.  General properties of oscillating integrals~\cite[Corollary~8]{Brown_toric} guarantee that this coincides with the $q$-series expansion of the stationary phase asymptotics of \eqref{eq:with_amplitudes} at the $\alpha$th critical point of $W$.  The key point here is that, for a generic value of $q$, if we let $\lambda_j \to 0$ for all $j$ along a generic path, then the critical points of $W$ corresponding to $\alpha \in F$ remain non-degenerate.  The stationary phase expansions at these critical points depend continuously (indeed analytically) on $\lambda$, and at $\lambda = 0$ remain well-defined.

Arranging the integrals \eqref{eq:with_amplitudes} into an $|F| \times |F|$ matrix and taking stationary phase asymptotics gives
\[
\left(\sum_{k=0}^\infty z^k \Psi_k \right) 
\begin{bmatrix}
  \ddots & & 0 \\
  &e^{w_\alpha/z} & \\
  0 & & \ddots
\end{bmatrix}
\]
where the factor on the right is a diagonal matrix, $w_\alpha$ is the value of $W$ at the $\alpha$th critical point, and $\Psi_k$ is an $|F| \times |F|$ matrix.
Here $\Psi_k$ and $w_\alpha$ depend analytically on $(q,t,\lambda)$ and are well-defined in the limit $\lambda=0$; also $\Psi_0$ is invertible, as a consequence of $\Delta_\beta(p_1,\ldots,p_k)$ forming a basis in $H^\bullet_T(X)$.  The right-hand side of \eqref{eq:entries_of_T} is obtained by replacing $\lambda_j$ here by $\lambda_j + z \partial_{\Lambda_j}$, and applying the resulting differential operator to $\partial_{\phi_b} J_B(\tau_B,z)$.  

For a function $\lambda \mapsto w(\lambda)$ on $H_2(B)$ that depends on parameters (such as $q_i$ and $t_i$), consider first the action of $e^{w(z \partial_\Lambda)/z}$ on $J_B(\tau_B,z)$.  By the Divisor Equation, the action of $z \partial_\Lambda$ on the $J$-function $J_B(\tau_B,z)$ coincides with the action of $\Lambda + z Q \partial_Q$ where $Q \partial_Q$ is the derivation of Novikov variables (for $B$) corresponding to $\Lambda \in H^2(B)$.  For each $D \in H_2(B)$, we have:
\[
e^{w(\Lambda + z Q \partial_Q)/z} Q^D = e^{w(\Lambda + zD)/z} Q^D
\]
and so $e^{w(\Lambda + z Q \partial_Q)/z}$ gives a well-defined operation on the space of cohomology-valued Laurent series in $z$ with coefficients that converge $Q$-adically, provided that $w(0) = 0$.  Due to the String Equation,
\[
e^{w(0)/z} J_B(\tau_B,z) = J_B\big(\tau_B + w(0)1,z\big)
\]
On the other hand $e^{w(\Lambda + z Q \partial_Q)/z - w(0)/z} J_B(\tau_B,z)$, after flipping the sign of $z$, lies in $z T_{J_B(\tau^*,{-z})}\cL_B \subset \cL_B$ for some point $\tau^* \in H^\bullet(B)$ that depends on $w$ and $\tau_B$.  This result was first used in~\cite{Coates-Givental} in the proof of the Quantum Lefschetz theorem, but was proved incorrectly there; an accurate proof is given in~\cite{CCIT} and~\cite[Theorem~1]{Givental:explicit}.

Applying these arguments with $w = w_\alpha$ for each $\alpha \in F$, we obtain
\[
e^{-w_\alpha(q,t,\lambda - z \partial_\Lambda)/z} J_B(\tau_B,{-z}) \in z T_{J_B(\tau_\alpha^*,{-z})}\cL_B \subset \cL_B
\]
for certain $\tau^*_\alpha \in H^\bullet(B)$ depending on $(q,t,\lambda)$ and $\tau_B$.  Differentiating with $\partial_{\phi_b}$ yields a basis, indexed by $b$, for $T_{J_B(\tau_\alpha^*,{-z})} \cL_B$ as a module over power series in $z$.  These bases together give a basis for the direct sum $\oplus_{\alpha \in F} T_{J_B(\tau_\alpha^*,{-z})}\cL_B$.
Note that applying $z \partial_\Lambda$ to a family of tangent vectors $v(\tau_\alpha^*) \in T_{J_B(\tau_\alpha^*,{-z})} \cL_B$ -- here the family depends on $\tau_B$ via $\tau_\alpha^*$ -- yields another family of tangent vectors in $T_{J_B(\tau_\alpha^*,{-z})} \cL_B$.  Therefore applying the $z$-series of matrix-valued differential operators $\sum_{k=0}^\infty (-z)^k \Psi_k(q,t,\lambda - z \partial_\Lambda)$ to our basis for $\oplus_{\alpha \in F} T_{J_B(\tau_\alpha^*,{-z})}\cL_B$ yields another basis for this direct sum.  This space, however, has a standard basis, formed by the columns of $S_B(\tau_\alpha^*,z)^{-1}$, $\alpha \in F$.  Expressing our basis in terms of the standard one, we obtain
\begin{multline}
  \label{eq:finally}
  \left(\sum_{k=0}^\infty z^k \Psi_k(q,t,\lambda - z \partial_\Lambda) \right) 
  \begin{bmatrix}
    \ddots & & 0 \\
    &e^{-w_\alpha(q,t,\lambda - z \partial_\Lambda)/z} & \\
    0 & & \ddots
  \end{bmatrix}
  S_B(\tau_B,z)^{-1} \\
  =
  \begin{bmatrix}
    \ddots & & 0 \\
    &S_B(\tau_\alpha^*,z)^{-1} & \\
    0 & & \ddots
  \end{bmatrix}
  R''(z)^{-1}
\end{multline}
for some invertible matrix-valued $z$-series $R''(z)$ with entries in the Novikov ring of $B$ that depend analytically on $(q,t,\lambda)$.  Here we used the fact that the columns of $S_B(\tau_B,z)^{-1}$ are $\partial_{\phi_b} J_B(\tau_B,{-z})$.  The left-hand side of \eqref{eq:finally} is the expansion near $z=0$ of $\Gamma_\block(z) \, T(-z)$, and the right-hand side provides its analytic extension to the non-equivariant limit $\lambda=0$.  Thus the  expansion near $z=0$ of $\Gamma_\block({-z}) \, T(z)$ has a well-defined non-equivariant limit, as claimed.

\subsection{Grading}
\label{sec:grading_preserving}
To complete the proof of Theorem~\ref{thm:main}, it remains to show that the loop group operator just defined respects the gradings.  In the non-equivariant limit $\lambda=0$, all the functions of $q$,~$t$,~$Q$, and~$\tau_B$ involved satisfy homogeneity conditions that reflect the natural grading in cohomology theory. To describe these conditions explicitly, introduce the {\em Euler vector field} 
\[ 
 \cE := 
\sum_{i=1}^k c_i q_i\frac{\partial}{\partial q_i} +
\sum_{a=1}^r \delta_a Q_a \frac{\partial}{\partial Q_a} +
\sum_{b=1}^{\text{rk}\, H^\bullet (B)} \left(1-\frac{\deg (\phi_b)}{2}\right) \tau_b \frac{\partial}{\partial \tau_b}.
\]
Here $c_i$ and $\delta_a$ are the coefficients of the first Chern class of $E$ with respect to an appropriate basis:
\[
c_1(E)=\sum_{j=1}^N U_j+\pi^*c_1(B) =
\sum_{i=1}^k\left(\sum_{j=1}^N m_{ij}\right) P_i +
\pi^*\left(c_1(B)-\sum_{j=1}^N\Lambda_j\right) = \sum_{i=1}^k c_iP_i+\sum_{a=1}^r
\delta_a \pi^*\phi_a
\]
where we have chosen our basis $\phi_1,\ldots,\phi_{\rk H^\bullet(B)}$ for $H^\bullet(B)$ such that $\phi_1,\ldots,\phi_r$ is a basis for $H^2(B)$.  

In what follows, we abuse notations and use the same notation to denote both an equivariant object and its non-equivariant limit. Let $\mu_E$ denote the Hodge grading operator for $E$. That is, in a homogeneous basis of $H^\bullet(E)$, $\mu_E$ is the diagonal matrix whose diagonal entries are half-integers expressing the degrees of the basis elements measured relative to the 
middle (complex) dimension of the target: 
\[
  \mu_E\big(\Delta_\beta(P)\otimes \pi^*\phi_b\big) =
  \frac{\text{deg}\Delta_\beta+\text{deg}\phi_b-\text{dim}_{\CC} E}{2}
  \big(\Delta_\beta(P) \otimes \pi^*\phi_b\big).
\]  

\begin{lemma}  
In its non-equivariant limit the matrix $T$, defined in \eqref{eq:what_is_T}, satisfies the grading condition
\begin{equation}
  \label{eq:grading_T}
  \left(z\frac{d}{dz}+\cE\right) T(-z)= T(-z) \mu_E-\mu_E T(-z),
\end{equation}
where $\mu_E$ is the Hodge grading operator for $E$.
\end{lemma}

\begin{proof} 
  Recall from \eqref{eq:what_is_I} that $I_E$ is a homogeneous expression of overall degree $1$ that takes values in the cohomology of $E$. Its components, in the graded basis $\Delta_\beta(P)\otimes \pi^*\phi_b$, have degrees $1-\deg \Delta_\beta /2 - \deg \phi_b /2 $ as functions of the variables $z, t, \tau, q, Q$. The columns of the matrix $T$ -- see \eqref{eq:what_is_T} -- are obtained from $I_E$ by applying the differential operators $\Delta_\beta(z\frac{\partial}{\partial t}) \partial_{\phi_b}$; these increase the degree by $\deg \Delta_\beta /2+ \deg \phi_b /2 -1$. The commutator $T(-z) \mu_E - \mu_E T(-z)$ on the right of \eqref{eq:grading_T} simply multiplies each matrix entry of $T$ by the difference of the column and row degrees, i.e. by the degree of the matrix entry as a homogeneous function of the variables $z, t, \tau, q, Q$. It equals the eigenvalue of this homogeneous function considered as an eigenvector of the Euler operator $z\frac{d}{dz}+\cE$ on the left hand side of \eqref{eq:grading_T}.
\end{proof}

We can rewrite the grading condition \eqref{eq:grading_T} using the Divisor Equations
\begin{align*}
  q_i\frac{\partial}{\partial q_i} T(-z) =
  \left(\frac{\partial}{\partial t_i} +\frac{P_i}{z}\right) T(-z) && \text{and} &&
 Q_a\frac{\partial}{\partial Q_a} T(-z) =
   \left(\frac{\partial}{\partial \tau_a} +\frac{\phi_a}{z}\right) T(-z)
\end{align*}
to replace $\cE$ with $\partial_{\cE} +c_1(E)/z$, 
where
\[ 
\partial_{\cE} = \sum_{i=1}^k c_i\frac{\partial}{\partial t_i} +
\sum_{a=1}^r \delta_a  \frac{\partial}{\partial \tau_a} +\sum_{b=1}^{\rk H^\bullet(B)}
\left(1-\frac{\deg (\phi_b)}{2}\right) \tau_b \frac{\partial}{\partial \tau_b}.
\]
From the non-equivariant version of the quantum differential equation \eqref{eqn:QDE} for $S_E$, we see that $T(-z)=S_E(z)^{-1}L(z)$ satisfies the ODE
\begin{align*}
  -z\partial_{\cE} T(-z) = T(-z) \cE (z)
  &&
  \text{where}
  && 
  \cE(z) = L^{-1}(z)(\partial_{\cE}\bullet) L(z) -z L^{-1}(z)\partial_{\cE} L(z)
\end{align*}
is regarded as a power series in $z$. Thus, considering $T$ as an operator, we have 
\begin{equation}\label{eq:commutation_T}
 \left( z\frac{d}{dz}+\mu_E+\frac{c_1(E)}{z}\right) T(-z) = T(-z)\left(
    z\frac{d}{dz} + \mu_E+\frac{\cE(z)}{z}\right).
\end{equation}

\medskip

On the other hand, the matrix $\cI$ of oscillating integrals, whose stationary phase asymptotics yield \eqref{eq:finally}, in the non-equivariant limit assumes the form
\[ 
\cI (-z):= \bigoplus_{\alpha \in F} \left[ \frac{\Delta_{\beta}\left(\textstyle -z \frac{\partial}{\partial t_1,} \dots,
-z \frac{\partial}{\partial t_k}\right)}{(-2\pi z)^{\frac{1}{2}\dim_{\CC} X}}
\int_{\gamma_\alpha} e^{-\sum_j x_j/z} \prod_j x_j^{\partial_{\Lambda_j}}
\frac{\prod_j d\log x_j}{\prod_i d\log (q_ie^{t_i})} \right]
S_B(\tau_B, z)^{-1}.
\]
Here columns of $\cI$ are indexed by pairs $(\beta, b)$ that correspond to elements of the basis $\{ \Delta_\beta(P) \otimes \pi^*\phi_b \}$ for $H^\bullet (E)$, while the rows are indexed by pairs $(\alpha, a)$ where $\alpha$ determines the integration cycle $\gamma_\alpha$ and $\phi_a$ is an element of the basis for $H^\bullet(B)$. Note that $a$ and $b$ are row and column indices in the matrix $S_B^{-1}$.

\begin{lemma} In its non-equivariant limit, the matrix $\cI$ satisfies the grading condition
\begin{equation}\label{eqn:grading_cI}
\left(z\frac{d}{dz}+\partial_{\cE}+\frac{c_1(B)}{z}\right) \cI(-z) = \cI (-z) \mu_E-\mu_B \cI (-z).
\end{equation}
\end{lemma}

\begin{proof} 
The Hodge grading operator $\mu_E$ on the right is decomposed as $\mu_X+\mu_B$ (or, more precisely, as $\mu_X \otimes 1+1\otimes \mu_B$ in the basis $\Delta_\beta(P)\otimes \pi^*\phi_b$ of $H^{\bullet}(E)$ indexing columns of $\cI$), and the part $\cI(-z) \mu_X$ arises here from the degrees of the operators $\Delta_\beta(-z\partial/\partial t)/(-2\pi z)^{\frac{1}{2}\dim_\CC X}$ with respect to $z\frac{d}{dz}$.   

We examine now how $z\frac{d}{dz}+\partial_{\cE}+\frac{c_1(B)}{z}$ acts on the integral in the definition of $\cI(-z)$. Since the Lie derivative of a closed form is exact, to differentiate an integral along a vector field it suffices to differentiate the integrand along a lift of the vector field to the domain of integration. The vector field $\sum_i c_i \partial/\partial t_i$ lifts along the projection \eqref{eq:fiber} to the Euler operator $Eu:=\sum_j x_j \partial/\partial x_j$. The latter acts trivially on $S_B^{-1}$ but satisfies 
$Eu \prod_jx_j^{\partial_{\Lambda_j}}=\prod_j x_j^{\partial_{\Lambda_j}}(Eu +\sum_j \partial_{\Lambda_j})$. 
Since $\sum_a \delta_a \partial/\partial \tau_a + \sum_j \partial_{\Lambda_j} = \partial_{c_1(B)}$, the action of $z \frac{d}{dz}+\partial_{\cE}+\frac{c_1(B)}{z}$ on the integral in $\cI(-z)$ amounts to applying $z\frac{d}{dz}+\partial_{c_1(B)}+\frac{c_1(B)}{z}+\sum_b \left(1-\frac{\deg (\phi_b)}{2}\right)\tau_b \frac{\partial}{\partial\tau_b}$ to $S_B^{-1}$ instead.

It remains to refer to the grading condition for $S_B^{-1}$ which, taking into account the Divisor Equations $Q_a\partial/\partial Q_a S_B^{-1}= (\partial /\partial \tau_a + \phi_a /z) S_B^{-1}$, $a=1,\dots, \rk H^2(B)$, assumes the form
\[ 
  \left(z\frac{d}{dz}+\partial_{c_1(B)}+\frac{c_1(B)}{z}+\sum_b \left(1-\frac{\deg (\phi_b)}{2}\right)\tau_b \frac{\partial}{\partial\tau_b}\right) S_B^{-1}= S_B^{-1} \mu_B - \mu_B S_B^{-1}.
\]
See \cite[Section 8]{Giv_quantization}. Altogether we obtain:
\[ \left(z\frac{d}{dz}+\partial_{\cE}+\frac{c_1(B)}{z}\right) \cI(-z) = \cI(-z)\mu_X+\cI(-z)\mu_B-\mu_B\cI(-z)\]
as promised.
 \end{proof}
 
As a function of $(t, \tau_B)$, $\cI$ satisfies the same differential equations
as $T$, since even before taking the non-equivariant limit $\cI$ and $T$ differ only by 
$\Gamma$-function factors that are independent of $t$ and $\tau_B$. In particular
\[ -z \partial_{\cE}\cI(-z)= \cI(-z) \cE(z).\]
Considering $\cI(-z)$ as an operator, we thus arrive at the following commutation relation:
\begin{equation}
\label{eq:commutation_I}
 \cI(-z)^{-1}\left( z\frac{d}{dz}+\mu_B +\frac{c_1(B)}{z}\right) =
 \left( z\frac{d}{dz}+\mu_E+\frac{{\cE}(z)}{z}\right) \cI (-z)^{-1}.
\end{equation}

Our previous results equate $\cD_E$, up to a constant factor, with
\[ \widehat{S_E^{-1}}\, \widehat{L} \, \widehat{R'}\,
\left(\widehat{\oplus_\alpha S_B(\tau^*_{\alpha})}\right)\, \cD_B^{\otimes |F|}.\]
Here $S_E^{-1}(z) L(z)$ coincides with $T(-z)$, and
$R'(z)\left(\oplus_{\alpha}S_B(\tau^*_{\alpha},z)\right)$ is the matrix inverse to the
stationary phase asymptotics of $\cI (-z)$.  Note that the values of the arguments $\tau$ and $\tau^*_{\alpha}$ in $S_E^{-1}$ and $S_B$ are determined by certain mirror maps and cannot be described here explicitly. Nevertheless the complicated relationship between them holds automatically, because all ingredients of the formula were constructed from the same function $I_E$. Now the commutation relations \eqref{eq:commutation_T} and \eqref{eq:commutation_I} show that commuting the Virasoro grading operator $l_0(E)=zd/dz +1/2+\mu_E+c_1(E)/z$ first across $T(-z)$ and then across the matrix inverse of the stationary phase asymptotics of $\cI(-z)$  yields $l_0(B)=zd/dz+1/2+\mu_B+c_1(B)/z$, the Virasoro grading operator for $B$. Thus the loop group transformation that relates $\cD_B^{\otimes |F|}$ and $\cD_E$ is grading-preserving, as claimed.


\section*{Acknowledgments}
We thank the referee for many valuable comments and suggestions. Coates was supported in part by a Royal Society University Research Fellowship, ERC Starting Investigator Grant number~240123, ERC Consolidator Grant 682603, and the Leverhulme Trust. Givental was supported by NSF grants DMS-0604705, DMS-1007164, DMS-1611839, and DMS-1906326. Tseng was supported in part by a Simons Foundation Collaboration Grant. Coates thanks the University of California at Berkeley for hospitality during the writing of this paper.  Givental thanks the Center for Geometry and Physics of the IBS at Pohang, Korea, and the Center's director Yong-Geun Oh for hospitality and support.  Special thanks are due to Jeff Brown, who was a part of this project since its inception, and even wrote the first draft of this paper, but later chose to withdraw from our team.


\begin{thebibliography}{10}

\bibitem{Audin}
Mich{\`e}le Audin.
\newblock {\em Torus actions on symplectic manifolds}, volume~93 of {\em
  Progress in Mathematics}.
\newblock Birkh\"auser Verlag, Basel, revised edition, 2004.

\bibitem{Behrend--Fantechi}
Kai Behrend and Barbara Fantechi.
\newblock The intrinsic normal cone.
\newblock {\em Invent. Math.}, 128(1):45--88, 1997.

\bibitem{BCFK}
Aaron Bertram, Ionu{\c{t}} Ciocan-Fontanine, and Bumsig Kim.
\newblock Two proofs of a conjecture of {H}ori and {V}afa.
\newblock {\em Duke Math. J.}, 126(1):101--136, 2005.

\bibitem{Brown_toric}
Jeff Brown.
\newblock Gromov-{W}itten invariants of toric fibrations.
\newblock {\em Int. Math. Res. Not. IMRN}, (19):5437--5482, 2014.

\bibitem{CCIT}
Tom Coates, Alessio Corti, Hiroshi Iritani, and Hsian-Hua Tseng.
\newblock Computing genus-zero twisted {G}romov-{W}itten invariants.
\newblock {\em Duke Math. J.}, 147(3):377--438, 2009.

\bibitem{Coates-Givental}
Tom Coates and Alexander Givental.
\newblock Quantum {R}iemann-{R}och, {L}efschetz and {S}erre.
\newblock {\em Ann. of Math. (2)}, 165(1):15--53, 2007.

\bibitem{Dijkgraaf--Witten}
Robbert Dijkgraaf and Edward Witten.
\newblock Mean field theory, topological field theory, and multi-matrix models.
\newblock {\em Nuclear Phys. B}, 342(3):486--522, 1990.

\bibitem{EHX}
Tohru Eguchi, Kentaro Hori, and Chuan-Sheng Xiong.
\newblock Quantum cohomology and {V}irasoro algebra.
\newblock {\em Phys. Lett. B}, 402(1-2):71--80, 1997.

\bibitem{EJX}
Tohru Eguchi, Masao Jinzenji, and Chuan-Sheng Xiong.
\newblock Quantum cohomology and free-field representation.
\newblock {\em Nuclear Physics B}, 510(3):608--622, 1998.

\bibitem{Getzler:Virasoro}
Ezra Getzler.
\newblock The {V}irasoro conjecture for {G}romov-{W}itten invariants.
\newblock In {\em Algebraic geometry: {H}irzebruch 70 ({W}arsaw, 1998)}, volume
  241 of {\em Contemp. Math.}, pages 147--176. Amer. Math. Soc., Providence,
  RI, 1999.

\bibitem{Getzler:jet}
Ezra Getzler.
\newblock The jet-space of a {F}robenius manifold and higher-genus
  {G}romov-{W}itten invariants.
\newblock In {\em Frobenius manifolds}, Aspects Math., E36, pages 45--89.
  Vieweg, Wiesbaden, 2004.

\bibitem{Giv_toric}
Alexander Givental.
\newblock A mirror theorem for toric complete intersections.
\newblock In {\em Topological field theory, primitive forms and related topics
  ({K}yoto, 1996)}, volume 160 of {\em Progr. Math.}, pages 141--175.
  Birkh\"auser Boston, Boston, MA, 1998.

\bibitem{Giv_quantization}
Alexander Givental.
\newblock Gromov-{W}itten invariants and quantization of quadratic
  {H}amiltonians.
\newblock {\em Mosc. Math. J.}, 1(4):551--568, 645, 2001.
\newblock Dedicated to the memory of I. G. Petrovskii on the occasion of his
  100th anniversary.

\bibitem{Giv_semisimple}
Alexander Givental.
\newblock Semisimple {F}robenius structures at higher genus.
\newblock {\em Internat. Math. Res. Notices}, (23):1265--1286, 2001.

\bibitem{Giv_Frob}
Alexander Givental.
\newblock Symplectic geometry of {F}robenius structures.
\newblock In {\em Frobenius manifolds}, Aspects Math., E36, pages 91--112.
  Friedr. Vieweg, Wiesbaden, 2004.

\bibitem{Givental:explicit}
Alexander Givental.
\newblock Explicit reconstruction in quantum cohomology and {$K$}-theory.
\newblock {\em Ann. Fac. Sci. Toulouse Math.}, (6) 25, no. 2--3, 419--432, 2016.

\bibitem{Graber-Pandharipande}
Tom Graber and Rahul Pandharipande.
\newblock Localization of virtual classes.
\newblock {\em Invent. Math.}, 135(2):487--518, 1999.

\bibitem{Iritani}
Hiroshi Iritani.
\newblock Convergence of quantum cohomology by quantum {L}efschetz.
\newblock {\em J. Reine Angew. Math.}, 610:29--69, 2007.

\bibitem{Joe-Kim}
Dosang Joe and Bumsig Kim.
\newblock Equivariant mirrors and the {V}irasoro conjecture for flag manifolds.
\newblock {\em Int. Math. Res. Not.}, (15):859--882, 2003.

\bibitem{Kont-Manin}
Maxim Kontsevich and Yuri Manin.
\newblock Gromov-{W}itten classes, quantum cohomology, and enumerative
  geometry.
\newblock {\em Comm. Math. Phys.}, 164(3):525--562, 1994.

\bibitem{Kontsevich--Manin}
Maxim Kontsevich and Yuri Manin.
\newblock Relations between the correlators of the topological sigma-model
  coupled to gravity.
\newblock {\em Comm. Math. Phys.}, 196(2):385--398, 1998.

\bibitem{Kont}
Maxim Kontsevich.
\newblock Intersection theory on the moduli space of curves and the matrix
  {A}iry function.
\newblock {\em Comm. Math. Phys.}, 147(1):1--23, 1992.

\bibitem{Kontsevich:enumeration}
Maxim Kontsevich.
\newblock Enumeration of rational curves via torus actions.
\newblock In {\em The moduli space of curves ({T}exel {I}sland, 1994)}, volume
  129 of {\em Progr. Math.}, pages 335--368. Birkh\"auser Boston, Boston, MA,
  1995.

\bibitem{Li--Tian}
Jun Li and Gang Tian.
\newblock Virtual moduli cycles and {G}romov-{W}itten invariants of algebraic
  varieties.
\newblock {\em J. Amer. Math. Soc.}, 11(1):119--174, 1998.

\bibitem{Liu-Tian}
Xiaobo Liu and Gang Tian.
\newblock Virasoro constraints for quantum cohomology.
\newblock {\em J. Differential Geom.}, 50(3):537--590, 1998.

\bibitem{OP}
Anderi Okounkov and Rahul Pandharipande.
\newblock Virasoro constraints for target curves.
\newblock {\em Invent. Math.}, 163(1):47--108, 2006.

\bibitem{US}
P.~Sankaran and V.~Uma.
\newblock Cohomology of toric bundles.
\newblock {\em Comment. Math. Helv.}, 78(3):540--554, 2003.
\newblock Errata: \emph{Comment. Math. Helv.}, 79(4):840--841, 2004.

\bibitem{Teleman}
Constantin Teleman.
\newblock The structure of 2{D} semi-simple field theories.
\newblock {\em Invent. Math.}, 188(3):525--588, 2012.

\bibitem{Witten}
Edward Witten.
\newblock Two-dimensional gravity and intersection theory on moduli space.
\newblock In {\em Surveys in differential geometry ({C}ambridge, {MA}, 1990)},
  pages 243--310. Lehigh Univ., Bethlehem, PA, 1991.

\end{thebibliography}
\end{document}